\documentclass{amsart}
\usepackage{amssymb}
\usepackage{amsthm}
\usepackage{amsmath}
\usepackage{amscd}
\usepackage{amsfonts}
\usepackage{enumerate}

\newcommand{\BR}{\mathbb{R}}
\newcommand{\BQ}{\mathbb{Q}}
\newcommand{\BZ}{\mathbb{Z}}
\newcommand{\BN}{\mathbb{N}}
\newcommand{\BA}{\mathbb{A}}

\newcommand{\D}{\mathcal{D}}

\newcommand{\Q}{\mathbb{Q}}
\newcommand{\ov}{\overline}

\newcommand{\gl}{\mathfrak{g}}

\newcommand{\s}{\mathfrak{s}}
\newcommand{\Ll}{\mathfrak{l}}
\def\k{{\Bbbk}}
\DeclareMathOperator{\Rad}{Rad}  

\renewcommand{\l}{\mathfrak{l}}

\newcommand{\Set}{\mathcal S}
\newcommand{\C}{\mathbb C}
\newcommand{\F}{\mathbb F}
\newcommand{\Z}{\mathbb Z}
\newcommand{\N}{\mathcal N}

\newcommand{\gP}{\mathfrak{P}}

\newcommand{\btri}{\blacktriangle}
\newcommand{\wtri}{\vartriangle}

\DeclareMathOperator{\uni}{uni}
\DeclareMathOperator{\nil}{nil}

\DeclareMathOperator{\Hom}{Hom}
\DeclareMathOperator{\End}{End}
\DeclareMathOperator{\Aut}{Aut}
\DeclareMathOperator{\Ad}{Ad}
\DeclareMathOperator{\Lie}{Lie}

\DeclareMathOperator{\GL}{GL}
\DeclareMathOperator{\SL}{SL}

\DeclareMathOperator{\Btype}{B}
\DeclareMathOperator{\Ctype}{C}

\DeclareMathOperator{\Gtype}{G}
\DeclareMathOperator{\Etype}{E}
\DeclareMathOperator{\Ftype}{F}
\DeclareMathOperator{\Tang}{T}

\DeclareMathOperator{\chara}{char}

\DeclareMathOperator{\im}{Im}
\DeclareMathOperator{\alg}{-alg}
\DeclareMathOperator{\ALG}{Alg}
\DeclareMathOperator{\SET}{Set}
\DeclareMathOperator{\GROUP}{Grp}

\newtheorem*{theorem}{Theorem}
\newtheorem*{corollary}{Corollary}
\newtheorem*{proposition}{Proposition}
\newtheorem*{lemma}{Lemma}

\theoremstyle{remark}

\newtheorem*{rmk}{Remark}
\newtheorem*{notation}{Notation}

\begin{document}
\title{The Hesselink stratification of nullcones and base change}
\author{MATTHEW C. CLARKE}
\address{Trinity College,
Cambridge, CB\textup{2 1}TQ, UK}
\email{m.clarke@dpmms.cam.ac.uk}

\author{ALEXANDER PREMET}
\address{School of Mathematics,
University of Manchester, Oxford Road, M\textup{13 9}PL, UK}
\email{sasha.premet@manchester.ac.uk}

\begin{abstract}
Let $G$ be a connected reductive algebraic group over an
algebraically closed field of characteristic $p \ge 0$. We give a
case-free proof of Lusztig's conjectures [Unipotent elements in
small characteristic, {\em Transform. Groups} 10 (2005), 449--487]
on so-called unipotent pieces. This presents a uniform picture of
the unipotent elements of $G$ which can be viewed as an extension of
the Dynkin--Kostant theory, but is valid without restriction on $p$.
We also obtain analogous results for the adjoint action of $G$ on
its Lie algebra $\gl$ and the coadjoint action of $G$ on $\gl^*$.

\end{abstract}

\maketitle

\section{Introduction and statement of results}  \label{intro}

\begin{notation}
In what follows $\k$ will denote an algebraically closed
field of arbitrary characteristic $p\ge 0$, unless stated otherwise.
Let $\gl$ denote the Lie algebra of $G$, and let $G_{\uni}$ and
$\gl_{\nil}$ denote the unipotent variety of $G$ and nilpotent
variety of $\gl$ respectively. By an $\s\Ll_2$-triple of $\gl$ we
mean elements $e,f,h \in \gl$ such that $\langle e,f,h \rangle \cong
\s\Ll_2(\k)$. We say that $p$ is {\em good} for $G$ if $p=0$
or $p$ is greater than the coefficient of the highest root in each
component of the root system of $G$, expressed as an integer
combination of simple roots. We denote by $G'$ a connected reductive
group over the complex numbers with the same root datum as $G$, and
$\gl'$ its Lie algebra. We use $\Hom(A,B)$ to denote the set of
algebraic group homomorphisms between algebraic groups $A$ and $B$,
and set $X(G)= \Hom(G,\k)$, and $Y(G)= \Hom(\k,G)$.
We use $\langle\ ,\ \rangle$ to denote the natural pairing $X(G)
\times Y(G) \rightarrow \BZ$. We let $G$ (resp. $\BZ$) act on $Y(G)$
by $g\cdot \lambda : \xi \mapsto g \lambda(\xi) g^{-1}$ (resp.
$n\lambda : \xi \mapsto  \lambda(\xi)^n $) for all $\xi \in
\k$. The identity element of $G$ will be denoted by $1_G$.
When $G$ acts on a set $X$, we let $X/G$ denote the set of
$G$-orbits in $X$. We use the convention that $\BN = \BZ_{\ge 0}$. If
$f : \k^{\times} \rightarrow V$ is a morphism of varieties and $v \in V$,
then we use the notation $\lim_{\xi \rightarrow 0}f(\xi) = v$ to mean that
$f$ may be extended to a morphism $\tilde{f} : \k \rightarrow V$ such that
$\tilde{f}(0) = v$.

\end{notation}


\subsection{} \label{Cparam} We begin by briefly reviewing some
classical results about unipotent elements of $G'$. First we assume
that $G'$ is a simple algebraic group of adjoint type over $\C$.
Springer has shown that there exists a $G'$-equivariant bijective
morphism $\sigma: G'_{\uni} \rightarrow \gl'_{\nil}$, a {\em
Springer morphism}. (Cf. \cite[Theorem III.3.12]{SpSt}. Usually the
group is required to be simply connected but in characteristic zero
the unipotent and nilpotent varieties of isogenous groups are
naturally isomorphic so we may drop that requirement in this case.)
Hence, the study of unipotent classes is equivalent to the study of
nilpotent orbits. Let $e \in \gl_{\nil}$. Then, by the
Jacobson--Morozov theorem, $e$ lies in an $\s\Ll_2$-triple of
$\gl'$. Kostant \cite{Kos} has shown that this induces a bijection
between $G'$-orbits of nilpotent elements and $G'$-orbits of
subalgebras of $\gl'$ isomorphic to $\s\Ll_2(\C)$. In \cite{Dyn}
Dynkin determined the latter in terms of characteristic diagrams
(now called weighted Dynkin diagrams), and showed that by
considering $\gl'$ as an $\s\Ll_2(\C)$-module, one can naturally
define an action of $\SL_2(\C)$ on $\gl'$. Thus, one obtains a
homomorphism of algebraic groups $\SL_2(\C) \rightarrow (\Aut
\gl')^{\circ} = G'$. Let
\begin{equation} \label{SL2factor} \tilde{D}_{G'} = \left\{ \omega \in Y(G') \ \left|   \begin{array}{c}
  \exists \ \tilde{\omega} \in \Hom(\SL_2(\C), G') \\
\mbox{ with } \omega(\xi) =  \tilde{\omega}\left[ \begin{smallmatrix}
\xi &  \\
 &  \xi^{-1} \end{smallmatrix} \right]
\end{array} \right.  \right\}.  \end{equation}
Then we have the following bijection of finite sets:
\begin{equation}  \label{Dynkinparam} \left\{ \mbox{unipotent classes of } G' \right\}
\stackrel{1-1}{\longleftrightarrow} \tilde{D}_{G'} \mbox{\bf \large /{\large $G'$}}.  \end{equation}

\noindent In fact (\ref{Dynkinparam}) holds even when we relax the assumption that
$G'$ is simple and adjoint, by well-known reduction arguments; see, e.g., \cite[Chapter 5]{Car2}.

\subsection{} Now assume that $p \ge 0$. It has been shown by
Springer and Steinberg in \cite{SpSt} that if $p > 3(h-1)$, where
$h$ is the Coxeter number of $G$, then everything described in the
previous subsection remains true, by essentially the same proofs.
Importantly, the analogue of $\tilde{D}_{G'}/G'$ for $p > 3(h-1)$ is
naturally in bijection with $\tilde{D}_{G'}/G'$, which can be seen
by identifying both with certain subsets of Weyl group orbits on one
parameter subgroups of a maximal torus. (We will consider a more
precise correspondence of one parameter subgroups attached to fixed
tori in Section \ref{scheme} by taking a scheme-theoretic approach.)
When $p \le 3(h-1)$ the $\s\Ll_2$-theory may no longer be available and
so an entirely different approach is necessary. However,
Pommerening's theorem (which extends the Bala--Carter theorem)
implies that, in fact, this parametrisation of unipotent classes
extends to any good $p$. This means that one may take
$\tilde{D}_{G'}/G'$ to be a parameter set for the unipotent classes
of any connected reductive group with the same Dynkin diagram as
$G'$, independent of good characteristic. More recently, a case-free
proof of Pommerening's theorem was found in \cite{Pre} and
simplified further in \cite{Tsu}. A Springer morphism also exists in
good characteristic and so $\tilde{D}_{G'}/G'$ also parametrises the
nilpotent orbits. Spaltenstein has shown further that this
parametrisation preserves the poset structure and dimensions of
classes, as well as certain compatibility relations between
parabolic subgroups, across different ground fields of good
characteristic (\cite[Th\'{e}or\`{e}me III.5.2]{Spa}).

When $p$ is a bad prime for $G$, the number of unipotent classes is
often greater than $|\tilde{D}_{G'}/G'|$, and, since Springer
morphisms do not exist when $p$ is bad, they need not be in
bijection with the nilpotent orbits. Both have been determined in
all cases, however. (See \cite[pp.~180--183]{Car2} for a
bibliographic account.) By a classical result of Lusztig
\cite{Lus0}, based on the theory of complex representations of
finite Chevalley groups, the orbit set $G_{\uni}/G$ is finite in all
characteristics. The orbit set $\gl_{\nil}/G$ is always finite as
well. Unfortunately, the only available proof of this fact for
groups of types ${\Etype}_7$ and ${\Etype}_8$ relies very heavily on
computer-aided computations; see \cite{HSp}. It turns out that in
all cases the cardinality of the set $G_{\uni}/G$ is less than or
equal to that of $\gl_{\nil}/G$.

\subsection{} \label{Gn}  Following \cite{Lus5} we now define
unipotent pieces. First note that $Y(G)/G$ is naturally isomorphic
to $Y(G')/G'$. (Indeed, in each case we may restrict to one
parameter subgroups of a fixed maximal torus, say $T$ and $T'$, since all maximal tori
are conjugate. Then the orbits are precisely the Weyl group orbits
on the $\Z$-modules $Y(T), Y(T')$, which can be identified
unambiguously.) We let $\tilde{D}_G$ denote the unique $G$-stable subset
of $Y(G)$ whose image in $Y(G')/G'$ corresponds to
$\tilde{D}_{G'}/G'$ under this bijection. Corresponding to
$\tilde{D}_G$ we define $D_G$ to be the set of sequences
\begin{equation*} \wtri = \left(G_0^{\wtri} \supset G_1^{\wtri} \supset G_2^{\wtri}
\supset \cdots \right) \end{equation*}

\noindent of closed connected subgroups of $G$ such that for some
$\omega \in \tilde{D}_G$ we have \begin{equation*} \Lie G_i^{\wtri}
= \,\left\{ x \in \gl \ \left|  \  \lim_{\xi \rightarrow 0} \xi^{1-i}(\Ad
\omega (\xi))x = 0 \right. \right\}. \end{equation*}

The obvious map $\tilde{D}_G \rightarrow {D}_G$ induces a bijection
$\tilde{D}_G/G \stackrel{\sim}{\rightarrow} {D}_G/G$ on the set of $G$-orbits.
Assume that $\omega \in \tilde{D}_G$ corresponds to some $G_0^{\wtri}$, and $T$
is a maximal torus of $G_0^{\wtri}$ containing $\im \omega$, and let $\Sigma$ denote
the root system of $G$ relative to $T$. Then one can show that \begin{equation*}   G_0^{\wtri} =
\left\langle T, U_{\alpha} \ |\ \alpha \in \Sigma,\  \langle \alpha, \omega \rangle \ge 0
\right\rangle, \mbox{ and }\end{equation*} \begin{equation*}   G_i^{\wtri} = \left\langle
U_{\alpha} \ |\ \alpha \in \Sigma,\  \langle \alpha, \omega \rangle \ge i \right\rangle
\mbox{ for $i \ge 1$ }, \end{equation*}

\noindent where the $U_{\alpha}$ are the root subgroups of $G$
relative to $T$. From this characterisation we see that
$G_0^{\wtri}$ is a parabolic subgroup of $G$, with unipotent radical
$G_1^{\wtri}$, and that $G_i^{\wtri}$ is normalised by $G_0^{\wtri}$
for any $i\ge 0$.

For any $G$-orbit $\btri \in D_G/G$, let $\tilde{H}^{\btri} =
\,\bigcup_{\wtri \in \btri} G_2^{\wtri}$. It is straightforward to
see that each set $\tilde{H}^{\btri}$ is a closed irreducible
variety stable under the conjugation action of $G$; see
Lemma~\ref{locally}. We now define
\begin{equation*}  {H}^{\btri}  :=  \,\tilde{H}^{\btri} \setminus
\textstyle{\bigcup}_{\btri'}\, \tilde{H}^{\btri'},
\end{equation*}
where the union is taken over all $\btri' \in D_G/G$ such that
$\tilde{H}^{\btri'} \subsetneqq \tilde{H}^{\btri}$. The subsets
${H}^{\btri}$ are called the {\em unipotent pieces} of $G$. We also
define \begin{equation*} {X}^{\wtri}  := \, G_2^{\wtri}\,
\textstyle{\bigcap}\, H^{\btri},      \end{equation*} for each
$\wtri \in D_G$, where $\btri$ is the $G$-orbit of $\wtri$. Since
${H}^{\btri}$ is the complement of finitely many non-trivial closed
subvarieties of $\tilde{H}^{\btri}$, it is open and dense in
$\tilde{H}^{\btri}$, hence it is locally closed in $G_{\uni}$. The
subset ${H}^{\btri}$ is $G$-stable since its complement in
$\tilde{H}^{\btri}$ is. Consequently, ${X}^{\wtri}$ is open and
dense in $G_2^{\wtri}$, and stable under conjugation by
$G_0^{\wtri}$. It is worth mentioning that $\btri\,\cong\,
G/G_0^\wtri$ as $G$-varieties.

\subsection{}\label{results}
In \cite{Lus5}, Lusztig has stated the following five properties and
conjectured that they should hold for all connected reductive groups
$G$ over algebraically closed fields.

\begin{enumerate}[$\gP_1.$]
    \item The sets $X^{\wtri}$ ($\wtri \in D_G$) form a partition of
    $G_{\uni}$, i.e. $G_{\uni}=\,\bigsqcup_{\wtri \in D_G}\,X^{\wtri}$.

    \smallskip

    \item For every $\btri \in D_G/G$ the sets $X^{\wtri}$ ($\wtri \in \btri$)
    form a partition of $H^{\btri}$.

    \smallskip

    \item The locally closed subsets $H^{\btri}$ ($\btri \in D_G/G$) form a (finite) partition of $G_{\uni}$,
    i.e.
    $G_{\uni}=\,\bigsqcup_{\btri \in D_G/G}\,H^{\btri}$.

    \smallskip

    \item For any $\wtri \in D_G$ we have that $G_3^{\wtri}X^{\wtri} = X^{\wtri}G_3^{\wtri} =
    X^{\wtri}$.

    \smallskip

    \item Suppose $\k$ is an algebraic closure of $\F_p$ and let
    $F\,\colon\,\, G \rightarrow G$ be the Frobenius endomorphism
    corresponding to a split $\F_q$-rational structure with $q-1$
    sufficiently divisible. Let $\wtri\in D_G$ be
    such that $F(G_i^{\wtri}) = G_i^{\wtri}$ for all $i \ge 0$ and let
    $\btri$ be the $G$-orbit of $\wtri\in D_G$. Then there exist polynomials
    $\varphi^{\btri}(t)$ and  $\psi^{\wtri}(t)$
    in $\BZ[t]$ with coefficients independent of $p$ such that $\varphi^{\btri}(q)=|H^{\btri}(\F_q)|$ and
    $\psi^{\wtri}(q)=|X^{\wtri}(\F_q)|$.
\end{enumerate}

\smallskip

When $p$ is good, properties $\gP_1$--$\,\,\gP_4$ follow from
Pommerening's classification; see \cite{Jan2}, \cite{Pom},
\cite{Pom2}. Lusztig has proved in \cite{Lus5}, \cite{Lus6} and
\cite{Lus7} that $\gP_1$--\,\,$\gP_5$ hold for classical groups (any
$p$) by a case-by-case analysis. For groups of type $\Etype$ (any
$p$) properties $\gP_1$--\,\,$\gP_5$ can be deduced from \cite{Miz},
although this is unsatisfactory since the extensive computations
which the results of that paper are based on are largely omitted,
and these results are known to contain many misprints. As mentioned
in \cite[p.~451]{Lus5} it is desirable to have an independent
verification of properties $\gP_1$--\,\,$\gP_5$ for groups of type
$\Etype$.

More recently Lusztig has introduced natural analogues of the
unipotent pieces $X^{\wtri}$ ($\wtri \in D_G$) and $H^{\btri}$
($\btri \in D_G/G$) for the adjoint $G$-module $\gl$ and its dual
$\gl^*$ and called them {\it nilpotent pieces} of $\gl$ and $\gl^*$.
Replacing $G_{\rm uni}$ by the nilpotent varieties $\N_\gl$ and
$\N_{\gl^*}$ (see Subsection \ref{HMum}) he conjectured that properties $\gP_1$--\,\,$\gP_5$
should hold for them as well. We stress that the $G$-modules $\gl$
and $\gl^*$  are {\it very} different when $p=2$ and $G$ is of type
$\Btype$, $\Ctype$ or $\Ftype_4$ and when $p=3$ and $G$ is of type
$\Gtype_2$. In all other cases there is a $G$-equivariant bijection
$\N_\gl\,\stackrel{\sim}{\rightarrow}\,\N_{\gl^*}$ which restricts
to a bijection between the corresponding nilpotent pieces and
induces a $1\,$--$\,1$ correspondence between the orbit sets
$\N_\gl/G$ and $\N_{\gl^*}/G$; see \cite[\S 5.6]{PS} for more details.
It is worth mentioning that the coadjoint action of $G$ on $\gl^*$
plays a very important role in studying irreducible
representations of the Lie algebra $\gl$.

In \cite{Lus6}, \cite{Lus7} and \cite{Lus1}, Lusztig proved that
$\gP_1$--\,\,$\gP_5$ hold for $\N_\gl$ in the case where $G$ is a
classical group and for $\N_{\gl^*}$ in the case where $G$ is a
group of type $\Ctype$. Very recently the coadjoint case for groups
of type $\Btype$ was settled by Ting Xue, a former PhD student of
Lusztig; see \cite{Xu}. In proving $\gP_1$--\,\,$\gP_5$ for
classical groups Lusztig and Xue relied on intricate counting
arguments involving linear algebra in characteristic $2$ and
combinatorics.

The main goal of this paper is to give a uniform proof of the
following using Hesselink's theory of the stratification of
nullcones.
\begin{theorem}\label{propP}
Let $G$ be a reductive group over an algebraically closed field of
characteristic $p \ge 0$ and $\gl=\Lie G$. Let $\mathcal G$ be one
of $G$, $\gl$ or $\gl^*$ and write $X^\wtri(\mathcal{G})$ for the
piece $X^\wtri$ of $\mathcal G$ labelled by $\wtri \in D_G$. Then
$\gP_1$--\,\,$\gP_5$ hold for $\mathcal G$ and the centraliser in
$G$ of any element in $X^\wtri(\mathcal{G})$ is contained in
$G_0^\wtri$.
\end{theorem}
We mention for completeness that the definition of nilpotent pieces
used by Lusztig and Xue for $G$ classical differs formally from
Lusztig's original definition in \cite{Lus5} which we follow.
However, Theorem~\ref{propP} implies that both definitions give rise
to the same partitions of $\N_\gl$ and $\N_{\gl^*}$; see
Remark~\ref{L-X} for more details. It is far from clear whether the
definition of Lusztig and Xue can be used for exceptional groups in
arbitrary characteristic.
\begin{rmk}
Regarding $\gP_2$, Lusztig has also conjectured that each piece
$H^{\btri}(\mathcal G)$ is a smooth variety and there exists a
$G$-equivariant fibration $f\colon H^{\btri}(\mathcal G)
\twoheadrightarrow \btri\cong G/G_0^\wtri$ such that
$f^{-1}(\wtri)\cong X^{\wtri}$ for all $\wtri \in\btri$. As far as
we know, the smoothness of $H^{\btri}(\mathcal G)$ is still an open
problem in bad characteristic. Using the techniques of
\cite[\S4]{bogomol} one can show that there always exists a
projective homogeneous $G$-variety $Y\cong G/P$, where $P$ is a
parabolic group scheme with $P_{\rm red}= G_0^\wtri$, and a
$G$-equivariant fibration $\varphi\colon H^{\btri}(\mathcal G)
\twoheadrightarrow Y$ whose fibres are isomorphic to $X^\wtri$ where
$\wtri\in \btri$. However, we do not know whether $\varphi$ can be
chosen to be separable, hence the smoothness of $H^{\btri}(\mathcal
G)$ is not guaranteed. On the other hand, in the Lie algebra case
there exist nilpotent pieces $H^\btri=H^{\btri}(\gl)$ which are not
$G$-equivariantly isomorphic to the geometric quotients
$G\times^{G^\wtri_0} X^\wtri$ with $\wtri\in\btri$. The simplest
example occurs when $\chara \k=2$, $G={\rm PSL}_2(\k)$ and
$X^\wtri=\,\k^\times e$ where $e$ is a nonzero nilpotent element of
$\gl=\,\mathfrak{pgl}_2(\k)$. To see this it suffices observe that
$\tilde{H}^\btri=\,\N_\gl=[\gl,\gl]$ is an abelian ideal of $\gl$ and hence
the derived action of $[\gl,\gl]\subset \Lie(G)$ on the the function algebra
$\k[H^\btri]\subset \k(\N_\gl)$ is trivial, whereas the action of
$[\gl,\gl]$ on $\k[G\times^{G^\wtri_0}
X^\wtri]=\,H^0(G/G_0^\wtri,\k[X^\wtri])$ is not trivial.
\end{rmk}
It is well-known that the sets $G_{\rm uni}$, $\N_\gl$ and
$\N_{\gl^*}$ coincide with the subvarieties of $G$-unstable elements
of the $G$-varieties $G$, $\gl$ and $\gl^*$, respectively (we assume
that $G$ acts on itself by conjugation). Therefore each set admits a
natural stratification coming from the Kempf--Rousseau theory, which
we review in Section~\ref{KRo}. In fact, such a stratification was
defined by Hesselink \cite{Hess2} for any affine $G$-variety $V$
with a distinguished point $*$ fixed by the action of $G$. It is
often referred to as the {\it Hesselink stratification} of the
variety of Hilbert nullforms of $V$. In Section~\ref{unip} we show
that every piece $H^{\btri}(\mathcal{G})$ coincides with a Hesselink
stratum of $\mathcal G$ and conversely every Hesselink stratum of
$\mathcal G$ has the form $H^{\btri}(\mathcal{G})$ for a unique
$\btri \in D_G/G$. We also identify the subsets
$X^{\wtri}(\mathcal{G})$ $(\wtri\in D_G)$ with the {\it blades} of
the variety of nullforms of $\mathcal G$. (As in the theorem we
assume here that $\mathcal G$ is one of $G$, $\gl$ or $\gl^*$.)

In order to relate the pieces $H^{\btri}(\mathcal{G})$ $(\btri \in
D_G/G)$ with Hesselink strata we first upgrade certain reductive
subgroups of $G$ involved in the Kempf--Ness criterion for
optimality of one parameter subgroups to reductive $\BZ$-group
schemes split over $\BZ$, and then make use of a well-known result of
Seshadri \cite{Sesh} on invariants of reductive group schemes. This is done in
Section~\ref{scheme}. After relating unipotent and
nilpotent pieces with Hesselink strata we deduce rather quickly that
$\gP_1$--\,\,$\gP_4$ hold for  $G$, $\gl$ and $\gl^*$.

\subsection{}\label{poly} Proving that $\gP_5$ holds for $G$, $\gl$ and $\gl^*$
requires more effort. Since our arguments involve induction on the
rank of the group we have to look at a much larger class of finite
dimensional rational $G$-modules.

Let $\mathfrak G$ be a reductive $\BZ$-group scheme split over $\BZ$
and suppose that $\k$ contains an algebraic closure of
$\F_p$. Set $G':={\mathfrak G}(\C)$ and $G:={\mathfrak
G}(\k)$. We say that a $G$-module $V$ is {\it admissible} if
there is a finite-dimensional $G'$-module $V'$ and an admissible
lattice $V'_\BZ$ in $V'$ such that $V=V_\BZ'\otimes_\BZ\k$. Recall
that a $\Z$-lattice in $V'$ is called {\it admissible} if it is
stable under the action of the distribution algebra ${\rm
Dist}_\BZ(\mathfrak G)$; see \cite{Jan} for more details. For any
$p^{th}$ power $q$ we may regard the finite vector space
$V(\F_q):=V'_\BZ\otimes_\BZ\F_q$ as an $\F_q$-form of the
$\k$-vector space $V$.

Since $G$ is a reductive group, the invariant algebra
$\k[V]^G$ is generated by finitely many homogeneous
polynomial functions $f_1,\ldots,f_m$ on $V$. The $G$-{\it nullcone}
of $V$, denoted $\N_{G,V}$ or simply  $\N_V$, is defined as the zero
locus of $f_1,\ldots,f_m$ in $V$. We set $\N_V(\F_q):=\N_V\cap
V(\F_q)$.
\begin{theorem}\label{null}
For every admissible $G$-module $V$ there exists a polynomial
$n_V(t)\in\BZ[t]$ such that $|\N_{V}(\F_q)|=n_V(q)$ for all $q=p^l$.
The polynomial $n_V(t)$ depends only on the $G'$-module $V'$, but
not on the choice of an admissible lattice $V'_\BZ$, and is the same
for all primes $p\in\mathbb{N}$.
\end{theorem}
In fact, a more general version of Theorem~\ref{null} is established
in Subsection~\ref{finitefield} which takes care of non-split
Frobenius actions on $G$. Property $\gP_5$ for $\N_\gl$ and
$N_{\gl^*}$ now follows almost at once since both $\gl$ and $\gl^*$
are admissible $G$-modules; see Section~\ref{dual}. Proving $\gP_5$
for $G_{\rm uni}$ requires some extra work; see
Corollary~\ref{last}. Theorem~\ref{null} enables us to show that the
classical results of Steinberg and Springer on the cardinality of
$G_{\rm uni}(\F_q)$ and $\N_\gl(\F_q)$, respectively, are
equivalent. It also enables us to compute the cardinality of
$\N_{\gl^*}(\F_q)$ thereby generalising a recent result of Lusztig
proved for $G$ classical; see \cite{Lus1} and \cite{Xu}.
\begin{corollary}\label{finite} Let $N=\dim G-{\rm rk}\,G$. Then
$|\N_\gl(\F_q)|=\,|\N_{\gl^*}(\F_q)|=q^N$ for any $p^{th}$ power $q$
and any prime $p\in\mathbb{N}$.
\end{corollary}
Once we observe that both $\gl$ and $\gl^*$ are admissible
$G$-modules coming from the adjoint $G'$-module $\gl'$,
Corollary~\ref{finite} becomes a consequence of Steinberg's formula
$|G_{\rm uni}(\F_q)|=q^N$  and the existence for $p\gg 0$ of a
$G$-equivariant isomorphism between $\N_\gl$ and $G_{\rm uni}$
defined over $\F_q$. Indeed, Theorem~\ref{null} then ensures that
the polynomial $n_\gl(t)=n_{\gl^*}(t)$ has coefficients independent
of $p$.

\smallskip

\noindent{\bf Acknowledgement.} The authors would like to thank
Anthony Henderson and George Lusztig for their comments on an
earlier version of this paper. Lusztig informed us that he has also
found a case-free proof of the formula $|\N_{\gl^*}(\F_q)|=q^N$
(unpublished). His idea was to show that the LHS is equal to the
number of $\F_q$-rational nilpotent elements in the Lie algebra of
the Langlands dual group and then use Springer's formula for that
case.

\section{The Kempf--Rousseau theory}\label{KRo}

Although much of this theory goes back to Mumford \cite{Mum}, Kempf
\cite{Kempf} and Rousseau \cite{Rous}, our set-up here is inspired
by Hesselink \cite{Hess2}, Slodowy \cite{Slod} and Tsujii
\cite{Tsu}.

\subsection{} \label{HMum} Let $V$ be a pointed $G$-variety, i.e. a $G$-variety with a distinguished
point $* \in V$ fixed by the action of $G$.  We will assume further
that $V$ is non-singular at $*$, although many results still hold
even when $*$ is singular. Let $H$ be a closed subgroup of $G$. Then
a point $v \in V$ is called $H$-{\em unstable} if there exists some
$\lambda \in Y(H)$ such that $\lim_{\xi \rightarrow 0}
\lambda(\xi)\cdot v = *$. Otherwise we say
that $v$ is $H$-{\em semistable}.

\begin{theorem} {\em (The Hilbert-Mumford criterion (cf.
\cite{MuFoKi}))} The following are equivalent.
\begin{enumerate}[{\em (i)}]
    \item  $v$ is $H$-unstable.
    \item  $f(v)=0$ for each regular function $f \in \k[V]^H$ which vanishes at $*$.
    \item  $0 \in \ov{H\cdot v}$.
\end{enumerate}  \end{theorem}

The set of all $G$-unstable elements is called the {\em nullcone},
denoted $\N_V$. It is well-known that $\k[V]^H$ is generated
(as a $\k$-algebra with $1$) by finitely many elements, and
so $\N_V$ is Zariski-closed in $V$. (In positive characteristic this
requires the Mumford conjecture proved by Haboush in \cite{Hab}.) If
we take $V = \gl$, with adjoint $G$-action and $*=0$, then in all
characteristics $\N_\gl = \gl_{\nil}$. Similarly, if $V = G$, with
the conjugation
action and $*=1_G$, then in all characteristics $\N_G = G_{\uni}$. 

\subsection{} Let $\psi : X \rightarrow Y$ be a morphism of affine varieties, and let
$\psi^* : \k[Y] \rightarrow \k[X]$ be its comorphism. Let $y \in Y$
and let $I_y$ be the maximal ideal of $y$ in $\k[Y]$. We define the
{\em coordinate ring} of the schematic fibre $\psi^{-1}(y)$ to be
$\k[X]/\psi^*(I_y)\k[X]$ (cf.  \cite[\S 14.3]{Eis}).  Now let $v \in
V$ and $\lambda \in Y(G)$. If $\lim_{\xi \rightarrow
0}\lambda(\xi)\cdot v = *$ and $v \not= *$, then the fibre of the
extended morphism at $*$ has coordinate ring $\k[T]/(T^m)$ for some
$m$, where $T$ is an indeterminate.

We now define a function which can be used to measure instability.
Given $\lambda \in Y(G)$ we define a function $m(-,
\lambda)\colon\,V\longrightarrow \Z_{\ge 0}\sqcup \{\pm\infty\}$ as
follows:
\begin{equation*} m(v,\lambda) :=\, \left\{ \begin{array}{ll}
         -\infty                &  \mbox{if $\lim_{\xi \rightarrow 0}\lambda(\xi)\cdot v$ does not exist} ;\\
         0                      &  \mbox{if $\lim_{\xi \rightarrow 0}\lambda(\xi)\cdot v = v' \not=*$} ;\\
         m \mbox{ (as above) }  &  \mbox{if $\lim_{\xi \rightarrow 0}\lambda(\xi)\cdot v = *$ \mbox{ ($v \not= *$) }};\\
         +\infty                &  \mbox{if $v = *$} .\end{array} \right.\end{equation*}

\noindent Note that $v \in V$ is $H$-unstable if and only if
$m(v,\lambda) \ge 1$ for some $\lambda \in Y(H)$. For a set $X \subset
V$ we also define $m(X, \lambda) = \inf_{v \in X} m(v,\lambda)$, and
say that $X$ is {\em uniformly unstable} if $m(X,\lambda) \ge 1$ for
some $\lambda \in Y(G)$.

\subsection{} \label{lambdasbgrps} Let $\lambda \in Y(G)$. We define some subgroups of
$G$ associated to $\lambda$ as follows: \begin{eqnarray*} P(\lambda)& := &
\left\{ g \in G \ \left|  \ \lim_{\xi \rightarrow 0} \lambda (\xi) g {\lambda
(\xi)}^{-1} \mbox{ exists } \right. \right\} , \\
L(\lambda) & := & C_G(\im \lambda), \\
U(\lambda) & := & \left\{ g \in G\ \left|  \ \lim_{\xi \rightarrow 0}
\lambda (\xi) g {\lambda (\xi)}^{-1} = 1_G \right. \right\}. \end{eqnarray*}

Let $T$ be a maximal torus of $L(\lambda)$ (and therefore a maximal
torus of $G$). If $\Sigma$ is the root system of $G$ relative to
$T$, then \begin{eqnarray*} P(\lambda) & = & \left\langle T,
U_{\alpha} \ \left| \ \alpha \in \Sigma,\ \langle \alpha,
\lambda \rangle \ge 0 \right. \right\rangle , \\
L(\lambda) & = & \left\langle T, U_{\alpha} \ \left| \ \alpha \in \Sigma,\ \langle \alpha,
\lambda \rangle = 0 \right. \right\rangle , \\
U(\lambda) & = & \left\langle U_{\alpha} \ \left| \ \alpha \in \Sigma,\ \langle \alpha, \lambda
\rangle \ge 1 \right. \right\rangle. \end{eqnarray*}

\noindent Hence $P(\lambda)$ is a parabolic subgroup of $G$ with unipotent radical $U(\lambda)$.
The following is now a straightforward exercise.

\begin{lemma} \label{basiclem} Let $v \in V$ and $\lambda \in Y(G)$. Then
$m(g\cdot v, \lambda) = m( v, g\cdot \lambda) = m( v, \lambda)$ for
all $g \in P(\lambda)$. In particular, for $i \ge 0$, the set of
$v\in V$ such that $m(v, \lambda) \ge i$ is
$P(\lambda)$-invariant.\end{lemma}

\subsection{} \label{KRthm} We define the set of {\it virtual one parameter subgroups} of $G$ as follows.
Let \begin{equation*}  Y_{\Q}(G)\, =\, (\BN \times Y(G)) / \sim,
\end{equation*}

\noindent where $\sim$ is the equivalence relation on $\BN \times
Y(G)$ such that $(n, \lambda) \sim (m, \mu)$ if and only if $n\mu =
m\lambda$. Note that $Y(G)$ is naturally a subset of $Y_{\Q}(G)$ and
the action of $G$ on $Y(G)$ naturally induces an action on $Y_{\Q}(G)$. If $T$ is a
torus, then $Y(T)$ is a free $\Z$-module, and so $Y_{\Q}(T) \cong \, Y(T)
\otimes_{\Z} \Q$ may be regarded as a $\Q$-vector space. We extend
our measure of instability to $Y_{\Q}(G)$ as follows. For $\lambda
\in Y_{\Q}(G)$, we have that $n \lambda \in Y(G)$ for some $n \in
\BN$ and so we may define $m(v, \lambda) = n^{-1} m(v, n\lambda)$.

A {\it squared norm mapping} on $Y_{\Q}(G)$ is a $G$-invariant
function $q : Y_{\Q}(G) \rightarrow \Q_{\ge 0}$ whose restriction to
$Y_{\Q}(T)$ for any maximal torus $T$ is a positive definite
quadratic form. By an averaging trick (cf. \cite[\S 7.1.7]{Sprin}) one
can always define a $W$-invariant positive definite quadratic form
$q$ on $Y_{\Q}(T)$. For an arbitrary $\lambda \in Y_{\Q}(G)$, let $g
\in G$ be such that $g \cdot \lambda \in Y_{\Q}(T)$. Then define
$q(\lambda)= q(g \cdot \lambda)$. One checks that this defines a
positive definite quadratic form on $Y_{\Q}(G)$ by observing that
the $G$ orbits on $Y_{\Q}(G)$ restrict to the $W$ orbits on
$Y_{\Q}(T)$. We define a map $\| \cdot \|_q : Y_{\Q}(G) \rightarrow
\BR_{\ge 0}$ by $\| \lambda \|_q\, := \sqrt{q(\lambda)}$ for all
$\lambda \in Y_{\Q}(G)$, which we call a {\it norm} on $Y_{\Q}(G)$.
From now on we will fix such a norm, and drop the subscript $q$. Let
$X \subset V$ and $\lambda \in Y(G) \setminus \{ 0 \}$. We say that
$\lambda$ is {\em optimal} for $X$ if \begin{equation*}
\frac{m(X,\lambda )}{\| \lambda \|} \ge \frac{m(X, \mu)}{\| \mu \|}
\ \mbox{ for all } \mu \in Y(G)\setminus \{0\}.  \end{equation*}

If $v \in V$ then, for ease of notation, we will often identify it
with the set $\{ v \}$ and thus talk about one parameter subgroups
which are optimal for $v$.  Usually the notion of optimality depends
on the norm, but in the special case that $V= \gl_{\nil}$ or
$G_{\uni}$, with adjoint or conjugation action respectively, or when
$V$ is a $G$-module, it is
independent of the norm by \cite[Theorem 7.2]{Hess}. Note that if
$\lambda$ is optimal for some set, then so is any non-zero scalar
multiple of $\lambda$. It will be convenient therefore to have a
canonical way of choosing an element in $(\mathbb{Q}^\times\lambda)
\cap Y(G)$ and for this we use the following notion from
\cite{Slod}. We say that $\lambda$ is {\em primitive} if we cannot
write $\lambda = n \mu$ for any integer $n \ge 2$ and $\mu \in
Y(G)$. If $X \subset V $ is uniformly unstable, we let $\Delta_X$
denote the set of all primitive elements in $Y(G)$ which are optimal
for $X$.

\begin{rmk} \label{Hessopt} Hesselink has defined a similar set in \cite{Hess2},
denoted $\Delta(X)$. This corresponds to a canonical choice for
optimal {\it virtual} one parameter subgroups. Let $\lambda \in
\Delta_X$. Then $\Delta(X) = \frac{1}{m(X,\lambda)}\Delta_X$. We
will need to use both sets later. To avoid confusion we will
use $\tilde{\Delta}_X$ to denote $\Delta(X)$, except in Subsection
\ref{modulestrata}, where it would be cumbersome to do so. \end{rmk}

\begin{theorem}  {\em (Kempf \cite{Kempf}, Rousseau \cite{Rous})} Let $X \subset V $ be
uniformly unstable.
\begin{enumerate}[{\em (i)}]
    \item We have $\Delta_X\not= \emptyset$ and there exists a parabolic subgroup $P(X)$ in $G$
    such that $P(X) = P(\lambda)$ for all $\lambda \in \Delta_X$.
    \item \label{vb} We have $\Delta_X = \{ g \cdot \lambda \ | \ g \in P(X) \}$ for any $\lambda \in \Delta_X$.
    \item If $T$ is a maximal torus of $P(X)$, then $Y(T) \cap \Delta_X$ contains exactly one element,
    which we denote by $\lambda_T(X)$.
    \item For any $g \in G$ we have that $\Delta_{g \cdot X} = g \Delta_X g^{-1}$ and
    $P(g \cdot X) = g P(X) g^{-1}$.
    The stabiliser $G_X = \{ g \in G\,|\,\,  g \cdot X = X \}$ is contained in $P(X)$. \end{enumerate}
    \end{theorem}  
\subsection{} \label{Slodowy} We now restrict to the special case where
$V$ is a finite-dimensional rational $G$-module with $* = 0$. Let
$T$ be a maximal torus of $G$ with Weyl group $W$. A very useful
set of tools for analysing the $T$-instability and optimality of subsets of $V$
are certain polytopes in $Y_{\Q}(T)$ defined in terms of weights of
the $T$-action on $V$. Let $X_{\Q}(T) = \,X(T) \otimes_{\BZ} \Q$,
and let $(\ \,,\ )$ be a $W$-invariant inner product on $Y_{\Q}(T)$
induced by the norm $\| \cdot \|$. Then there is a $\Q$-linear
isomorphism $\phi_T \colon\, X_{\Q}(T) \rightarrow Y_{\Q}(T)$
defined uniquely by the relation $\langle \chi , \lambda \rangle = (
\phi_T(\chi) , \lambda )$ for all $\chi \in X_{\Q}(T)$ and $\lambda
\in Y_{\Q}(T)$.

Consider the weight space decomposition $V = \bigoplus_{\chi \in
X(T)} V_{\chi}$ of $V$ with respect to $T$, where
\begin{equation*} \label{semiinv} V_{\chi} = \left\{ v \in V  \ | \
t \cdot v = \chi (t) v \mbox{ for all } t \in T \right\}.
\end{equation*}
Then for any $v \in V$ we may uniquely write $v = \sum_{\chi \in
X(T)} v_{\chi}$ with $v_{\chi} \in V_{\chi}$. If $X \subset V$, we
define $S_T(X) :=\, \{ \chi \in X(T)\, |\,\,  v_{\chi} \not= 0
\mbox{ for some } v \in X\}$, and let $K_T(X)$ denote the convex
hull (or Newton polytope) of $\phi_T(S_T(X))$ in $Y_{\Q}(T)$. Then
we have the following.
\begin{lemma}  {\em (Cf. \cite{Slod})}\label{opti} Let $X \subset V$ and $T$ be a maximal torus of $G$.
\begin{enumerate}[{\em (i)}]
    \item If $\lambda \in Y(T)$, then $m(X, \lambda) = \,\min_{\mu \in \phi_T(S_T(X))}\, (\mu, \lambda)
    =\,
    \min_{\mu \in K_T(X)}\, (\mu, \lambda).$
    \item There exists a unique element $\mu_T(X) \in K_T(X)$ of minimal norm.
    \item \label{normsquared} The set $X$ is uniformly $T$-unstable if and only if $\mu_T(X) \not= 0$,
    in which case we have that $\| \mu_T(X) \|^2 = \,m(X, \mu_T(X))$.
    \item \label{uniqueintorus} If $X$ is $T$-unstable and $\lambda_T(X)$ is the unique primitive
    scalar multiple of $\mu_T(X)$, then $\Delta_{X, T} = \{\lambda_T(X)\}$.   \end{enumerate} \end{lemma}
\subsection{} \label{HessIsom} Resume the more general assumption that
$V$ is a $G$-variety. For $i \ge 0$ and $\lambda \in Y_{\Q}(G)$, we
denote by $V(\lambda)_i$ be the set of elements $v \in V$ with $m(v,
\lambda) \ge i$, a closed subvariety of $V$. Let $X \subset V$ be
uniformly unstable and suppose that $\lambda \in \Delta_X$ and $k =
m(X,\lambda)$. Then we define the {\em saturation} of $X$ to be
$S(X) = V(\Delta_X)_k$. This is well-defined by Theorem
\ref{KRthm}(\ref{vb}) and Lemma \ref{basiclem}. We call a set {\em
saturated} if it is uniformly unstable and equal to its own
saturation.

Assume, temporarily, again that $V$ is a $G$-module with $* = 0$.
We may grade $V$, with respect to $\lambda$, as a direct sum of
subspaces \begin{equation*} V(\lambda,i) =\, \left\{ v \in V
\,|\,\, \lambda(\xi)\cdot v = \xi^{i}v \mbox{ for all } \xi
 \in \k^{\times} \right\},
\end{equation*}
for $i \in \BZ$.
Then a saturated set $X \subset V$ may be written as
\begin{equation*} X =\, V(\Delta_X)_k = \,\textstyle{\bigoplus}_{i \ge k}\,
V(\lambda, i),\end{equation*} where $\lambda \in \Delta_X$ and $k =
m(X,\lambda)$. Letting $T$ be a maximal torus of $C_G(\im \lambda)$,
it is not hard to see that the $V(\lambda, i)$ are sums of weight
subspaces of $V$. More precisely,
$$X = \textstyle{\bigoplus}_{\langle \chi, \lambda \rangle \ge k}\,
V_{\chi}.$$ Since all maximal tori of $G$ are conjugate and $V$ has
finitely many $T$-weights, the number of conjugacy classes of
saturated subsets of $V$ is finite.

The following result of Hesselink shows that the description
of
saturated sets in the general situation, in which $V$ is a
$G$-variety, may be reduced to the above consideration. (Note that
since $*$ is $G$-invariant, the tangent space $\Tang_*V$ naturally
becomes a $G$-module.)

\begin{proposition}  {\em (Hesselink \cite[Proposition
3.8]{Hess2})}\label{linear} If $X$ is a saturated subset of $V$,
then $\Tang_*X$ is a saturated subset of $\Tang_*V$ which is isomorphic to
$X$ and satisfies $\Delta_{\Tang_*X} = \Delta_{X}$. The application of
$\Tang_*$ is a bijection from the class of saturated subsets of $V$ to
the class of saturated subsets of $\Tang_*V$. \end{proposition}

In particular, the saturated sets in the adjoint action of $G$ on
itself are connected unipotent subgroups.

By virtue of Proposition \ref{HessIsom} we may implicitly identify a
saturated set with its tangent space, so that Lemma \ref{Slodowy}
now makes sense for arbitrary saturated sets. We now gather some
basic facts about saturated sets that will be useful later. First we
need the following definitions. Given a uniformly $G$-unstable
subset $X$ of $V$ we define
$$\| X \| := \,\min\big\{\| \mu_T(g\cdot X) \|\,\colon\,\,g\in\ G,\ \,0\not\in K_T(g\cdot X)\}.$$
Note that $\| X \|$ is the minimal distance from the origin to a
point in a finite union of polytopes of the form $K_T(g\cdot X)$ for
{\it some} $g\in G$, and it is independent of the choice of $T$. It
follows from Lemma~\ref{opti} that
$\|X\|\,=\inf\{\|\lambda\|\,\colon\,\,\lambda\in Y(G),\
m(X,\lambda)\ge 1\}$ (cf. \cite{Hess2}, p.~143).

\begin{lemma} \label{saturated}  Let $X$ and $Y$ be uniformly unstable subsets of $V$.
\begin{enumerate}[{\em (i)}]
\item \label{sati} $S(X)$ is uniformly unstable, $\Delta_{S(X)} = \Delta_X$ and
$\tilde{\Delta}_{S(X)} = \tilde{\Delta}_X$.
\item \label{satii} $\tilde{\Delta}_X = \tilde{\Delta}_Y$ if and only if
$Y \subset S(X)$ and $\| X \| = \| Y \|$.
\item \label{satiii} $X \subset S(X) = S(S(X))$.
\item \label{sativ} If $X \subset Y$, then $\| X \| \ge \| Y \|$.
\item \label{satv} If $g \in G$, then $g \cdot S(X) = S(g \cdot X)$.\end{enumerate}   \end{lemma}
\begin{proof} This is a straightforward exercise. Cf. \cite[Lemma 2.8]{Hess2}. \end{proof}

\subsection{} \label{strata} Following \cite[\S 4]{Hess2} now define
some equivalence relations on $\N_V$. For $x, y \in \N_V$ we set
\begin{eqnarray*}
x \approx  y     &  \Leftrightarrow & \tilde{\Delta}_x = \tilde{\Delta}_y; \\
x \sim     y     &  \Leftrightarrow & \tilde{\Delta}_{g\cdot x} = \tilde{\Delta}_y
\mbox{ for some } g \in G. \end{eqnarray*}
 \noindent We call an equivalence class $[ v ] = \{ x\,|\,\, x \approx v \}$ a
 {\em blade} and an equivalence class $G[ v ] = \{ x\,|\,\, x \sim v \}$ a {\em stratum}.
 Hesselink gives the following description of blades and strata.

\begin{lemma} \label{bladestrataprop} Let $v \in \N_V$. Then \begin{enumerate}[{\em (i)}]
\item \label{bsi} $[v] = \{ x \in S(v)\,\colon\,\,  \| x \| = \| v \| \}$.
\item \label{bsii} $[v]$ is open and dense in $S(v)$.
\item \label{bsiii} $GS(v)$ is an irreducible closed subset of $\N_V$.
\item \label{bsiv} $G[v]\, =\, \{ x \in GS(v)\,\colon\,\,  \| x \| = \| v \| \}$.
\item \label{bsv} $G[v]$ is open and dense in $GS(v)$.
\end{enumerate}   \end{lemma}

We will eventually show that when $V=G_{\uni}$ the strata are
precisely Lusztig's unipotent pieces. To that end the following
result will be crucial.
\begin{proposition} \label{stratasetminus}
Let $v \in V$. Then
\begin{equation*}  G[v] \, = \, GS(v) \setminus \textstyle{\bigcup}\, GS(v'),      \end{equation*}
where the union is taken over all saturated sets $S(v')$ such that
$GS(v') \subsetneqq GS(v)$. \end{proposition}
\begin{proof} Let $v,v'\in \N_V$ be such that $GS(v')
\subseteq GS(v)$. In order to prove the proposition, it is
sufficient to show that $GS(v') = GS(v)$ if and only if $\| v \| =
\| v' \|$.

Suppose that $GS(v') = GS(v)$. Then there exists $g \in G$ such that
$g \cdot v' \in S(v)$. Hence $\| v' \| = \| g \cdot v' \| \ge \| S(v) \| = \| v
\|$ by Lemma \ref{saturated}. Similarly we can find $h \in G$ such
that $h \cdot v \in S(v')$ and deduce that $\| v' \|  \le \| v \|$, and
thus $\| v' \| =  \| v \|$.

Conversely, suppose that $\| v' \|  =  \| v \|$. Since $GS(v')
\subseteq GS(v)$, there exists $g \in G$ such that $g \cdot v' \in S(v)$.
Then Lemma \ref{saturated}(\ref{satii}) yields $\tilde{\Delta}_{g \cdot v'}
= \tilde{\Delta}_{v}$, and so $S(g \cdot v') = S(v)$. Hence $g \cdot S(v') = S(v)$
by Lemma \ref{saturated}(\ref{satv}). It follows that $GS(v') =
GS(v)$. \end{proof}

\section{A modification of the Kirwan--Ness theorem}\label{KirNess}
\subsection{} Let $\lambda \in Y(G) \setminus \{ 0 \}$ and let $T$ be a maximal torus of
$G$ containing $\im \lambda$. (This is equivalent to $T$ being a
maximal torus of $L(\lambda)$.) Then we define \begin{eqnarray*}
T^{\lambda} & := & \left\langle \im \mu  \ |\ \mu \in Y(T), \ (\mu,
\lambda ) = 0 \right\rangle, \\  L^{\perp}(\lambda ) & := &
\left\langle T^{\lambda}, \ \D L(\lambda) \right\rangle.
\end{eqnarray*}
Note that $L^{\perp}(\lambda )$ is independent of the choice of $T$
since $(gTg^{-1})^{\lambda} = gT^{\lambda}g^{-1}$ for all $g \in G$.
Also, $T^{\lambda}$ is a subtorus of $T$ and $L^{\perp}(\lambda ) =
T^{\lambda} \cdot \D L(\lambda)$ is a connected reductive group by
\cite[Corollary 2.2.7]{Sprin}, \cite[\S IV.14.2]{Bor}.
\subsection{} \label{Tsujii} We now restrict to the special case where $V$
is a $G$-module with $* = 0$. In \cite{Slod}, \cite{PV}, \cite{Tsu} the following
generalisation of the Kirwan--Ness theorem is proved.
\begin{theorem}    {\em (Cf. Kirwan \cite{Kir}, Ness \cite{Ness})}
Let $v \in V \setminus \{ 0 \}$ and $\lambda \in Y(G) \setminus \{ 0 \}$. Assume
that $k=m(v, \lambda) \ge 1$ and write $v = \sum_{i \ge k} v_i$ with
$v_i \in V(\lambda, i)$ (and $v_k\ne 0$). Then $\lambda$ is optimal
for $v$ if and only if $v_k$ is $L^{\perp}(\lambda)$-semistable.
\end{theorem}

 Our goal is to obtain an analogous result for the conjugation action of
 $G$ on the unipotent variety. Our proof is modelled on the proof in
 \cite{Tsu} of the above result. We will need the following lemmas from
 \cite{Slod} and \cite{Tsu} for this task.

\subsection{} \label{orthproj} We continue to assume that $V$ is a $G$-module
with $* = 0$. It follows from \cite[Proposition 8.2(c)]{Bor} that an
element of $X_{\Q}(T^{\lambda})$ may be lifted to an element of
$X_{\Q}(T)$. In fact, $X_{\Q}(T^{\lambda})$ may be naturally
identified with the orthogonal projection of  $X_{\Q}(T)$ onto the
hyperplane $\{\chi \in X_{\Q}(T) \, | \,\, (\chi, \lambda ) = 0 \}$.
The following lemma shows that this projection behaves well with
respect to optimality.
\begin{lemma} {\em (Cf. \cite{Slod})} \label{perp} Let $\lambda \in Y(G) \setminus \{ 0 \}$
and $v \in V(\lambda, k)$ for some $k \in \BN$. If $T$ is a maximal
torus of $G$ containing $\im \lambda$ then $\mu_{T^{\lambda}}(v) =
\mu_T(v)  -  \frac{k}{(\lambda, \lambda)} \lambda $. \end{lemma}

\subsection{} \label{lambdamincomp} We continue to assume that $V$ is
a $G$-module with $* = 0$. The following is the key lemma used
in the proof of Theorem \ref{Tsujii}.

\begin{lemma} {\em (\cite[Lemma 2.6]{Tsu})}\label{tsu} Let $T$ be a maximal
torus of $G$ and assume that $v \in V \setminus \{ 0 \}$ is $T$-unstable.
Let $k = m(v, \lambda_T(v))$ and $v' \in v + \bigoplus_{i>k}
V(\lambda_T(v), i)$. Then $\lambda_T(v) = \lambda_T(v')$.
\end{lemma}

\subsection{}\label{groupcase} We now assume that $V = G_{\uni}$ with $* = 1_G$. Let
$\lambda \in Y(G)$ and let $T$ be a maximal torus of $L(\lambda)$
with corresponding $G$-root system $\Sigma$. Recall that for each root $\alpha \in
\Sigma $ we denote the corresponding root subgroups by $U_{\alpha}$,
and we have that \begin{equation*} R_u(P(\lambda)) =\,U(\lambda) :=
\,\left\langle U_{\alpha} \, |\,\, \alpha \in \Sigma,\ \langle
\alpha, \lambda \rangle \ge 1 \right\rangle,
\end{equation*}
where $R_u(P(\lambda))$ denotes the unipotent radical of
$P(\lambda)$. In fact, $U(\lambda )$ is directly spanned by the root
subgroups $U_{\alpha}$ with $\langle \alpha, \lambda \rangle \ge 1$;
see \cite[\S IV.14]{Bor}. Hence the product morphism
\begin{equation*} \pi : U_{\alpha_1} \times U_{\alpha_2} \times
\cdots \times U_{\alpha_n} \longrightarrow\,
\textstyle{\prod}_{\langle \alpha, \lambda \rangle \ge
1}\,U_{\alpha}\, = \,U(\lambda)
\end{equation*} is an isomorphism of varieties, with respect to any
choice of ordering $\{ \alpha_1, \dots, \alpha_n \}  =\,  \{ \alpha
\in \Sigma \, |\, \, \langle \alpha, \lambda \rangle \ge 1 \}$,
which we now fix once and for all. Moreover, since each root
subgroup $U_{\alpha}\,=\,\langle x_\alpha(t)\,|\,\,t\in\k\rangle$
is isomorphic to the additive group $\k^{+}$, this gives an
isomorphism $f : U(\lambda) \stackrel{\sim}{\longrightarrow}\,
\BA^n(\k)$. Consider $\BA^n(\k)$ as a vector space
with basis indexed by the set $\{1,2,\ldots, n\}$. It becomes a
$T$-module by letting $t \in T$ act on the $i^{th}$ basis vector by
scalar multiplication by $\alpha_i(t)$. With respect to this $f$ is
$T$-equivariant. From now on we will implicitly regard $U(\lambda)$
as a $T$-module.

We define the following $L(\lambda)$-stable closed subvarieties of
$U(\lambda)$ for each $i \ge 1$: Let $ \{ \beta_{1}, \beta_{2},
\dots, \beta_{l(i)} \}\,  = \, \{ \alpha \in \Sigma \, | \,\,
\langle \alpha, \lambda \rangle = i \},$ and set
\begin{equation*} U^i(\lambda)  \,=\, \pi(U_{\beta_{1}} \times
U_{\beta_{2}} \times \cdots \times U_{\beta_{l(i)}}).
\end{equation*}
These give a direct product decomposition of $U(\lambda)$ into
$T$-submodules, and we may identify
\begin{equation*} U(\lambda)\,   \cong\,  U^1(\lambda) \times
U^2(\lambda) \times \cdots \times U^r(\lambda), \end{equation*} for
some $r\in\BN$, so that for any $u \in U(\lambda)$ we may uniquely
write $\pi^{-1}(u) = (u_1,u_2,\, \ldots, \,u_r )$ with $u_i \in
U^i(\lambda)$. For $\lambda \not= 0$ and $u \not= 1_G$ define
$m'(u,\lambda) :=\, \min \{ i \, | \,\, u_i \not= 1_G\}$ and
 $m'(u,\lambda) :=\, +\infty$ for $u = 1_G$. Then we have
 the following.
\begin{lemma}\label{m(u)}
Let $\lambda \in Y(G) \setminus \{ 0 \}$ and $u \in U(\lambda )$.
Then $m'(u,\lambda) = m(u,\lambda)$. \end{lemma}
\begin{proof}  If $u=1_G$, the statement is obvious, so suppose $u\ne 1_G$.
For each root $\alpha_i$ let $m_i = \langle \alpha_i, \lambda
\rangle$. Then we have a morphism of varieties $\ell :
\BA^1(\k) \rightarrow U(\lambda)$ given by
$t\longmapsto\,\lambda(t) u \lambda(t)^{-1}$ for $t \in
\k^{\times}$ and $\ell(0) = 1_G$. Writing $u =
\pi^{-1}(u_{\alpha_1},u_{\alpha_2},\, \ldots, \,u_{\alpha_n} )$ with
$u_{\alpha_i}=x_{\alpha_i}(\xi_i) \in U_{\alpha_i}$, we have
\begin{eqnarray*} \ell(t) & = & \pi^{-1}\big(\lambda(t) u_{\alpha_1}
\lambda(t)^{-1},\, \lambda(t) u_{\alpha_2} \lambda(t)^{-1},\,\ldots\, ,\,
\lambda(t) u_{\alpha_n} \lambda(t)^{-1} \big) \\
        & = & \pi^{-1}\big( x_{\alpha_1}(\xi_1t^{\langle \alpha_1,\,  \lambda
        \rangle}),\, x_{\alpha_2}(\xi_2 t^{\langle \alpha_2,\, \lambda
        \rangle})
        ,\, \ldots\,,\, x_{\alpha_n}(\xi_n t^{\langle \alpha_n,\,
        \lambda \rangle})\big) \\
        & = & \pi^{-1}\big( x_{\alpha_1}(\xi_1t^{m_1}),\,
        x_{\alpha_2}(\xi_2t^{m_2}),\,  \ldots\,,\,
        x_{\alpha_n}(\xi_nt^{m_n})\big).
        \end{eqnarray*}
Without loss of generality assume that $m_1 \le m_2 \le \cdots \le
m_n$ and $m'(u, \lambda) = m_k$ for some $k\le n$, so that $\xi_i=0$
for $i<k$. Then, identifying $\k[U(\lambda )]$ and
$\k[\BA^1(\k)]$ with the polynomial rings
$\k[T_1, \dots, T_n]$ and $\k[T]$ respectively, the
comorphism $\ell^*$ sends $g = g(T_1, \dots, T_n ) \in
\k[U(\lambda )]$ to $g(0,\ldots, 0,\xi_kT^{m_k}, \ldots,\,
\xi_nT^{m_n})$. Hence, if $I = \langle T_1, \dots, T_n \rangle$ is
the maximal ideal of $1_G \in U(\lambda)$, then the ideal
$\ell^*(I)$ of the schematic fibre $\ell^{-1}(u)$ is generated by
$\xi_kT^{m_k}, \ldots,\,  \xi_nT^{m_n}$. As $\xi_k\ne 0$, it follows
that the coordinate ring of the schematic fibre $\ell^{-1}(u)$
equals $\k[T] / (T^{m_k})$.

Now consider the composition $\BA^1(\k)
\stackrel{\ell}{\longrightarrow}\, U(\lambda)
\stackrel{\iota}{\longrightarrow}\, G_{\uni}.$ If $\iota(1_G)= 1_G$
has maximal ideal $I'$ of $\k[G_{\uni} ]$, then $\iota^*(I')
= I$, so that $(\iota \circ\, \ell )^*(I') =\, \ell^* \circ\,
\iota^*(I') =\, \ell^*(I)$, which completes the proof.
\end{proof}
\subsection{}\label{KNess}
For $i\ge 1$, we set $U_i(\lambda) := \,\left\langle U_{\alpha} \,
|\,\, \alpha \in \Sigma,\ \langle \alpha, \lambda \rangle \ge i
\right\rangle,$ a connected normal subgroup of $U(\lambda)$. The
group $L(\lambda)$ acts rationally on the affine variety
$V_i(\lambda)\,:=\, U_i(\lambda)/U_{i+1}(\lambda)\,\cong\,
U^i(\lambda)$. The variety $V_i(\lambda)$ is a connected abelian
unipotent group. It may be regarded as a vector space over $\k$ with
basis $v_1,\ldots, v_{l(i)}$ consisting of the images of
$x_{\beta_1}(1),\ldots, x_{\beta_{l(i)}}(1)$ in
$U_i(\lambda)/U_{i+1}(\lambda)$. Our convention here is that
$\xi_1v_1+\cdots+\xi_{l(i)}v_{l(i)}$ is the image of
$\prod_{j=1}^{l(i)}\,x_{\beta_j}(\xi_j)$ in
$U_i(\lambda)/U_{i+1}(\lambda)$ for all $\xi_i\in\k$. The preceding
remarks then imply that the torus $T\subset L(\lambda)$ acts
linearly on $V_i(\lambda)\,\cong\, U^i(\lambda)$ with the $v_j$
being weight vectors of $V_i(\lambda)$ with respect to $T$. In view
of Chevalley's commutator relations it is straightforward to see
that each root subgroup $U_\alpha$ with $\langle \alpha, \lambda
\rangle =0$ acts linearly on $V_i(\lambda)$ as well. It follows that
the group $L(\lambda)$ acts linearly and rationally on
$V_i(\lambda)$. In other words, each vector space $V_i(\lambda)$ is
a rational $L(\lambda)$-module.

We are now ready to state and prove the following version of the
Kirwan--Ness theorem.

\begin{theorem} \label{KNuni}
Let $u \not= 1_G$ be a unipotent element of $G$ and $\lambda \in
Y(G) \setminus \{ 0 \}$. Assume that $u \in U(\lambda)$ and let $k= m(u,
\lambda)$. Then $\lambda $ is optimal for $u$ if and only if the
image of $u$ in $V_k(\lambda)=U_k(\lambda)/U_{k+1}(\lambda)$ is
$L^{\perp}(\lambda)$-semistable.
\end{theorem}
\begin{proof} In proving the theorem we may assume without loss of generality that
$\lambda$ is primitive. We follow Tsujii's arguments from
\cite[Theorem 2.8]{Tsu} very closely.

First suppose $\lambda$ is optimal for $u$ and let $k=m(u,\lambda)$.
Then $u\in U_k(\lambda)\setminus U_{k+1}(\lambda)$ by
Lemma~\ref{m(u)}. Let $\bar{u}$ denote the image of $u$ in the
$L^\perp(\lambda)$-module
$V_k(\lambda)\,=\,U_k(\lambda)/U_{k+1}(\lambda)$. We must show that
$\bar{u}$ is semistable with respect to all maximal tori of
$L^{\perp}(\lambda)$. Of course, each of these has the form
$T^{\lambda}$ for some maximal torus $T$ of $L(\lambda)$. In
particular, $\lambda\in Y(T)$ and hence $\lambda=\lambda_T(u)$ by
our assumption on $\lambda$. Note that Lemma~\ref{opti} can be used
in our present (non-linear) situation in view of
Proposition~\ref{linear} applied with $G\,=\,T$. Then $k =
(\mu_T(u), \lambda_T(u))$, so that
$$\mu_T(u)\in\{\mu\in K_T(u)\,|\,\, (\mu, \lambda_T(u)) = k\} \,=\,
K_T(\bar{u}).$$ Therefore $\mu_T(u) = \mu_T(\bar{u})$ and
$\lambda_T(u) = \lambda_T(\bar{u})$. Let $\mu\in Y(T ) \setminus \{ 0 \}$.
Then Lemma~\ref{opti} implies that
$$\frac{m(u,\lambda_T(u))} {\|\lambda_T(u)\|}\, =\,
\frac{k}{\|\lambda_T(u)\|}\,
=\,\frac{m(\bar{u},\lambda_T(u))}{\|\lambda_T(u)\|}\,=
\,\frac{m(\bar{u},\lambda_T(\bar{u}))}{\|\lambda_T(\bar{u})\|}\,\ge\,
\frac{m(\bar{u},\mu)}{\|\mu\|}.$$ Since $S_T(\bar{u})\subseteq
S_T(u)$ we have that $m(\bar{u},\mu)\ge m(u,\mu)$. Then
$\lambda_T(\bar{u})\in\Delta_{T,u}\,=\,\{\lambda_T(u)\}$, implying
that $\mu_{T^\lambda}(\bar{u})$ and $\lambda$ are proportional; see
Lemma~\ref{perp}. Since $\lambda$ is orthogonal to
$\mu_{T^\lambda}(\bar{u})\in Y(T^\lambda)$ it must be that
$\|\mu_{T^\lambda}(\bar{u})\|=0$. Hence $\bar{u}$ is
$T^\lambda$-semistable by Lemma~\ref{opti}(iii).

Conversely, suppose that $\bar{u}$ is $L^\perp(\lambda)$-semistable.
The parabolic subgroups $P(\lambda)$ and $P(u)$ have a maximal torus
in common, $T'$ say; see \cite[Corollary 28.3]{Hum}. We may choose
$w \in U(\lambda)$ with $T:=wT'w^{-1} \subset L(\lambda)$ so that
$\lambda \in Y(T)$. Then $\bar{u}$ is $T^{\lambda}$-semistable by
the assumption and hence $\mu_{T^{\lambda}}(\bar{u}) = 0$ by Lemma
\ref{Slodowy}. Applying Lemma \ref{orthproj} we now get
$\mu_T(\bar{u}) = \frac{k}{(\lambda,\lambda)}\lambda$. It follows
that $\lambda = \lambda_{T}(\bar{u})$. We claim that also $\lambda =
\lambda_{T}(wuw^{-1})$.

In order to prove the claim we first recall that $U(\lambda)$ has a
$T$-module structure such that
$U_i(\lambda)/U_{i+1}(\lambda)\,\cong\, U^i(\lambda)$ as $T$-modules
for all $i\ge 1$; see Subsection~\ref{groupcase}. Then
$\lambda_{T}(\bar{u})=\lambda_{T}(u_k)$. In view of
Lemma~\ref{lambdamincomp}, we need to show that the $k$-component of
$wuw^{-1}$ is $u_k$ (which will then be the minimal non-trivial
component of $wuw^{-1}$, by Lemma~\ref{lambdasbgrps}). Write $u =
\prod_{\langle \alpha, \lambda \rangle \ge k } u_{\alpha}$ and
assume that $w=\prod_{i=1}^n x_{\alpha_i}(\zeta_i)$ for some
$\zeta_i\in\k$.  Then Chevalley's commutator relations yield
\begin{eqnarray*} wuw^{-1}  &=&  \mathop{\prod_{\alpha \in
\Sigma}}_{\langle \lambda, \alpha \rangle \ge k} \!\!\!
wu_{\alpha}w^{-1}  \in   \mathop{\prod_{\alpha \in \Sigma}}_{\langle
\lambda, \alpha \rangle \ge k} \!\!\! \Big(u_{\alpha} \!\!\!
\mathop{\prod_{i, j > 0}}_{i\alpha + j\beta \in \Sigma} \!\!\!
U_{i\alpha + j\beta}\Big)\\ & \subseteq & \Big(\mathop{\prod_{\alpha
\in \Sigma}}_{\langle \lambda, \alpha \rangle \ge k} \!\!\!
u_{\alpha} \Big)\cdot U_{k+1}(\lambda) \ \,  \subseteq\,\  u_k U_{k+1}(\lambda).
\end{eqnarray*}
Hence $\lambda = \lambda_{T}(wuw^{-1})$ as claimed. To complete the
proof of the theorem note that $T \subset wP(\lambda)w^{-1} =
P(wuw^{-1})$, and so $\lambda \in \Delta_{wuw^{-1}} = \Delta_u$ by
Theorem \ref{KRthm}.  \end{proof}
\begin{rmk}\label{overZ}
For each $\beta\in\Sigma$ with $\langle\beta,\lambda\rangle=k$ we
let $v_\beta$ denote the image of $x_\alpha(1)$ in
$V_k(\lambda)\,=\,U_k(\lambda)/U_{k+1}(\lambda)$ and write $X_\beta$
for the tangent vector of the root subgroup $U_\beta=\,\langle
x_\beta(t)\,|\,\,t\in\k\rangle$ in $\gl\,=\,\Lie G$, so that
$$(\Ad x_\beta(t))\,y\,\equiv\,y+t[X_\beta,y]\quad\ \,\big({\rm
mod}\,\,\gl\otimes t^2\k[t]\big)\qquad\quad\,
(\forall\,\,y\in\,\gl\otimes\k[t]).$$ The map
$v_\beta\mapsto\,X_\beta$ extends uniquely up to a linear
isomorphism between $V_k(\lambda)$ and the subspace
$\gl(\lambda,k)=\,{\rm
span}\,\{X_\beta\,|\,\,\langle\beta,\lambda\rangle=k\}$; we call it
$\eta_k$. Using Chevalley's commutator relations and our definition
of the vector space structure on $V_k(\lambda)$ at the beginning of
this subsection it is straightforward to see that $\eta_k$ is an
isomorphism of $L(\lambda)$-modules. If $G$ and $T$ are defined over
$\BZ$, then so is $\eta_k$.
\end{rmk}
\section{Reductive group schemes and Seshadri's theorem} \label{scheme}
We now briefly review reductive group schemes before stating a
result of Seshadri which we will need later. For a general reference
see \cite{Jan}, for example.
\subsection{}  For an affine variety $X$ over $\k$, we say
that $X$ is {\em defined over} $\BZ $ if there is an embedding of
$X$ into some affine space $\BA^n(\k)$ such that the radical
ideal $I(X)$ of $X$ is generated by elements of $\BZ [X_1, \dots,
X_n]$. (This is the same as requiring that $\k[X]\cong \BZ
[X] \otimes_{\BZ} \k$, where $\BZ [X]= \BZ [X_1, \dots,
X_n]/(I(X) \cap \BZ [X_1, \dots, X_n])$.) A morphism  $\phi : X
\rightarrow Y$ of $\k$-varieties defined over $\BZ$ is said
to be defined over $\BZ$ if it can be written in terms of elements
of $\BZ [X_1, \dots, X_n]$. (This is the same as requiring that its
comorphism restricts to a homomorphism $\phi^* : \BZ [Y] \rightarrow
\BZ [X]$ of $\BZ$-algebras.)

When $X$ is defined over $\BZ$ we may associate to it a reduced
affine algebraic $\BZ$-scheme, i.e. a functor $\mathfrak{X} :
\ALG_{\BZ} \rightarrow \SET$ such that if $A, A'$ are $\BZ$-algebras
and $\psi : A \rightarrow A'$ is a $\BZ$-algebra homomorphism then
$\mathfrak{X}(A) = \Hom_{\BZ\alg}(\BZ[X], A)$ and
$\mathfrak{X}(\psi): \alpha \mapsto \psi \circ \alpha$ for each
$\alpha \in \Hom_{\BZ\alg}(\BZ[X], A)$. We identify
$\mathfrak{X}(A)$ with the set $\{ a \in A^n \ | \ f(a) = 0 \mbox{
for all } f \in I(X) \cap A [X_1, \dots, X_n] \}$.

If $G$ is an affine algebraic group over $\k$, then we say
that $G$ is defined over $\BZ$ if it is so as a variety and the
product and inverse morphisms are defined over $\BZ$. (This is the
same as requiring that the Hopf algebra structure on $\k[G]$
restricts to one on $\BZ [G]$.) In this case we may associate to it
(using Jantzen's terminology) a reduced algebraic $\BZ$-group, i.e.
a functor $\mathfrak{G} : \ALG_\k \rightarrow \GROUP$
defined as above, with the group structure on $\mathfrak{G}(A)$
defined via the Hopf algebra structure on
$A[G]=\BZ[G]\otimes_{\BZ}A$ for each $\BZ$-algebra $A$. From now on
we call such a functor a {\em $\BZ$-group scheme}. $G$ is said to be
$\BZ$-{\em split} if there exists a maximal torus $T$ of $G$ such
that there is an isomorphism $T \stackrel{\sim}{\rightarrow}
\k^{\times} \times \cdots \times \k^{\times}$ which
is defined over $\BZ$ and the root morphisms of $T$ are defined over
$\BZ$.

It has been shown by Chevalley (\cite{Chev}) that every connected
reductive algebraic group over an algebraically closed field $\k$
may be obtained by extension of scalars from a reduced algebraic
$\BZ$-group, and that many familiar subgroups and actions are also
defined over $\BZ$. This allows one to pass information between the
characteristic zero and prime characteristic settings; see
\cite{Jan}. We will use this to relate optimal one parameter
subgroups of reductive groups $G$ in arbitrary characteristic to
those of reductive groups $G'$ with the same root system defined
over $\C$. This will eventually allow us to use the parameter set
$\tilde{D}_{G'}/G'$ from Section~\ref{intro} in arbitrary
characteristic.
\subsection{} \label{Seshadri} Let $\mathfrak{G}$ be a reductive $\BZ$-group
scheme and let $\mathfrak{X}$ be a reduced affine algebraic
$\BZ$-scheme. We will say that $\mathfrak{G}$ {\em acts} on $\mathfrak{X}$
if, for any $\BZ$-algebra $A$, there is a map $\phi_A :
\mathfrak{G}(A) \times \mathfrak{X}(A) \rightarrow \mathfrak{X}(A)$,
functorial in $A$, given by polynomials over $A$. If $\mathfrak{G}$
acts on an affine space $\BA^n_{\BZ}$ (regarded as a $\Z$-scheme)
then we say that this action is linear if, for any $\BZ$-algebra
$A$, $g \in \mathfrak{G}(A)$, the map $\phi_A(g) : \BA^n_{\BZ}(A)
\rightarrow \BA^n_{\BZ}(A)$ is $A$-linear.

We now state a result of Seshadri (\cite{Sesh}) which allows one to
pass information about semistability between characteristics.
\begin{theorem}  {\em (Cf. \cite[Proposition 6]{Sesh})} Let $\k$
be an algebraically closed field and let $\mathfrak{G}$ be a
reductive $\BZ$-group scheme acting linearly on $\BA^n_{\BZ}$.
Suppose that $\mathfrak{X}$ is a $\mathfrak{G}$-stable open
subscheme of $\BA^n_{\BZ}$ and $x \in \mathfrak{X}(\k)$ is a
semistable point. Then there exists a $\mathfrak{G}$-invariant $F
\in \BZ[ \BA^n_{\BZ} ] = \BZ [X_1, \dots, X_n]$ such that $F(x)
\not= 0$. Furthermore, there is an open subscheme
$\mathfrak{X}^{ss}$ of $\mathfrak{X}$ such that for any
algebraically closed field $\k'$, the set
$\mathfrak{X}^{ss}(\k')$ consists of the semistable points
of $\mathfrak{X}(\k')$. \end{theorem}
\subsection{} \label{rootdatahom}
In the next section we will prove our main result by applying
Theorem \ref{Seshadri} to a reductive $\BZ$-group scheme associated
with $L^{\perp} (\lambda)$. To that end we will now construct such a
scheme. From now on assume that we have a fixed reductive
$\BZ$-group scheme $\mathfrak{G}$, which determines the reductive
groups $G, G'$ that we are interested in. In addition, let us fix a
maximal torus $\mathfrak{T}$ of $\mathfrak{G}$. Then there is a
natural identification of the one parameter subgroups of
$\mathfrak{T}(\k)$ as $\k$ varies. It follows that
there is a reductive $\BZ$-group scheme $\mathfrak{L}$, the
scheme-theoretic centraliser of a one parameter subgroup $\lambda$
of $\mathfrak{T}$, which gives rise to the groups $L(\lambda)$. The
groups $L^{\perp}(\lambda)$ may also be obtained from a reductive
$\BZ$-group scheme, but since this is not a standard result we will
now give an explicit construction.

Recall that a root datum of a connected reductive group, or
reductive $\BZ$-group scheme, is a quadruple $(X(T), \Sigma, Y(T),
\Sigma^{\vee})$, with respect to a fixed maximal torus, together with the
perfect pairing $X(T) \times Y(T) \rightarrow \Z$ and the associated
bijection $\Sigma \rightarrow \Sigma^{\vee}$ between the roots and coroots of
$G$ with respect to $T$. If we forget about the fixed torus $T$ and
merely regard $X(T)$ and $Y(T)$ as abstract free abelian groups with
finite subsets $\Sigma$ and $\Sigma^{\vee}$ respectively, then the datum is
unique and moreover any such abstract root datum gives rise to a
connected reductive group, or reductive group $\Z$-scheme. If $G'$
is another such group, or $\Z$-group scheme, with datum $(X(T'), \Sigma',
Y(T'), \Sigma'^{\vee})$, then a {\em homomorphism of root data} is a
group homomorphism $f: X(T') \rightarrow X(T)$ that maps $\Sigma'$
bijectively to $\Sigma$ and such that the dual homomorphism $f^{\vee}: Y(T)
\rightarrow Y(T')$ maps $f(\beta)^{\vee}$ to $\beta^{\vee}$ for each
$\beta \in \Sigma'$. A morphism of algebraic groups $\psi : T
\rightarrow T'$ is said to be {\em compatible with the root data} if
the induced homomorphism $\psi^* : X(T') \rightarrow X(T)$ is a
homomorphism of root data.
\begin{proposition}\label{Zscheme}
The connected reductive group $L^{\perp}(\lambda)$ is a $\Z$-scheme
theoretic subgroup of $L(\lambda)$. In other words, if
$\mathfrak{L}$ is a $\Z$-group scheme such that $\mathfrak{L}(\k) =
L(\lambda)$, then there exists a $\Z$-subgroup scheme
$\mathfrak{L}^{\perp}$ of $\mathfrak{L}$ such that
$\mathfrak{L}^{\perp}(\k) = L^{\perp}(\lambda)$.
\end{proposition}
\begin{proof} Suppose that $(X(T), \Sigma, Y(T), \Sigma^{\vee})$ is the root
datum of $L(\lambda)$. It follows then that the root datum of
$L^{\perp}(\lambda)$, with respect to the maximal torus
$T^{\lambda}$, is $\big(X(T^{\lambda}),\,
\{\alpha\vert_{T^{\lambda}}\, | \,\, \alpha \in \Sigma \},\,
Y(T^{\lambda}),\, \Sigma^{\vee}\big)$. We may also construct reductive
$\BZ$-group schemes from these data, say $\mathfrak{L}$ (as above)
for the former and $\tilde{\mathfrak{L}}^{\perp}$ for the latter. We
now need to construct a subgroup scheme $\mathfrak{L}^{\perp}$ of
$\mathfrak{L}$, isomorphic to $\tilde{\mathfrak{L}}^{\perp}$ which
gives rise to $L^{\perp}(\lambda)$. We start by showing that
$T^{\lambda}$ is defined over $\BZ$ as a subgroup of $T$, so that we
may construct a $\BZ$-group scheme $\mathfrak{T}$ with subgroup
scheme $\mathfrak{T}^{\lambda}$ which give rise to $T$ and
$T^{\lambda}$ respectively.

We know that $T^\lambda$ is a subtorus of codimension $1$ in T (for
it is a connected subgroup of $T$ and $Y(T^\lambda)$ has rank equal
to $l-1$ where $l=\dim T$). Therefore $T/T^\lambda$ is
a $1$-dimensional torus. By \cite[Corollary 8.3]{Bor} the natural
short exact sequence $1\rightarrow T^\lambda\rightarrow T\rightarrow
T/T^\lambda\rightarrow 1$ gives rise to a short exact sequence of
character groups $0\rightarrow X(T/T^\lambda)\rightarrow
X(T)\rightarrow X(T^\lambda)\rightarrow 0.$ Since $T/T^\lambda$ is a
one dimensional torus, its character group $X(T/T^\lambda)$ is
generated by one element, say $\eta$. By the above $\eta$ can be
regarded as a rational character of $T$ and
\begin{equation}\label{eta}X(T)\cong\, \BZ\eta\oplus
X(T^\lambda).\end{equation} (One should keep in mind here that
$X(T^\lambda)$ is a free $\BZ$-module of rank $l-1$.) By
construction, $\eta$ vanishes on $T^\lambda$.

On the other hand, \cite[Proposition~8.2(c)]{Bor} shows that
$T^\lambda$ coincides with the intersection of the kernels of
rational characters of $T$, say $T^\lambda=\bigcap_{\chi\in A} {\rm
ker}\, \chi$ where $A$ is a non-empty subset of $X(T)$. If $A$
contains a character of the form $a\eta+\mu$ for some non-zero
$\mu\in X(T^\lambda)$ then $T^\lambda\subseteq {\rm ker}\, \eta\cap
{\rm ker}\, \mu$. But then $\dim T^\lambda \le l-2$ because $\eta$
and $\mu$ are linearly independent in $X_\BQ(T)$. Since this is
false, it must be that $A\subseteq \BZ\eta$. As a result,
$T^\lambda={\rm ker}\,\eta$.

The above argument is characteristic-free since $\eta$ can be
described as the unique, up to a sign, primitive element of $X(T)$
proportional to $\lambda$ in $X_\BQ(T)$, which we identify with
$Y_\BQ(T)$ by means of our $W$-invariant inner product. In view of
(\ref{eta}) we may regard $\eta$ as one of the standard generators
of the Laurent polynomial ring $\C[T]$. This implies that $\eta-1\in
\BZ[T]$ generates a prime ideal of $\C[T]$, thus showing that
$T^\lambda={\rm ker}\,\eta$ is defined over $\BZ$. This enables us
to construct the desired subgroup scheme $\mathfrak{T}^{\lambda}$ of
$\mathfrak{T}$.

The inclusion $\mathfrak{T}^{\lambda} \subset \mathfrak{T}$ induces
a homomorphism of root data, and by \cite[Proposition II.1.15]{Jan}
(and the proof) there exists an injective homomorphism of
$\BZ$-group schemes $\iota : \tilde{\mathfrak{L}}^{\perp}
\hookrightarrow \mathfrak{L}$ which agrees on the root subgroups. We
may therefore take $\mathfrak{L}^{\perp}$ to be the functor defined
by $A \mapsto \iota(\tilde{\mathfrak{L}}^{\perp})(A)$ for any
$\BZ$-algebra $A$. We know that this gives rise precisely to
$L^{\perp}(\lambda)$ since the restriction of the functor $\iota$ to
the root subgroups determines it uniquely by \cite[II.1.3(10)]{Jan}.
\end{proof}
\section{Unipotent pieces in arbitrary characteristic}\label{unip}
\subsection{} \label{Kraft} We will need the following result, due to H. Kraft,
during the proof of our next theorem. This was not published by
Kraft but the details can be found in \cite{Hess}; see Theorem 11.3
and the remarks in \S 12. Let $(e,h,f)$ be an $\s\l_2$-triple of
$\gl'$ and assume that we have the usual grading on $\gl'$ given by
$\gl'(i) = \{x \in \gl' \ | \  [h,x] = ix \}$ for all $i \in \Z$. Let $\rho
: \mathbb{C}^{\times} \rightarrow ({\rm Aut}\,\gl')^\circ$ be
defined by $\rho(\xi)x = \xi^ix$ if $x \in \gl'(i)$. It follows that
there is a one parameter subgroup $\lambda' \in Y(G')$ such that
$\rho = \Ad \circ \lambda'$. We then say that $\lambda'$ is {\em
adapted to} $e$. (For full details see \cite[\S E, p.~238]{SpSt}.) If
$\nu \in \Hom(\SL_2(\C),G')$, then we define $\nu_* \in Y(G')$ by
composing $\nu$ with the map $\xi \mapsto \left[ \begin{smallmatrix}
\xi &  \\
 &  \xi^{-1} \end{smallmatrix} \right]$.

\begin{theorem} {\em (H. Kraft, unpublished)} The following are true.
\begin{enumerate}[{\em (i)}]
    \item Let $e \in \gl'_{\nil}$ and assume that $\lambda' \in Y(G')$ is a
    one parameter subgroup adapted to $e$. Then $\frac{1}{2}\lambda' \in
    \tilde{\Delta}_e$.

  \item Let $u \in G'_{\uni}$ and assume that we have $\nu \in \Hom(\SL_2(\C), G')$ such that
  $\nu \left[ \begin{smallmatrix}
1 & 1 \\
 &  1 \end{smallmatrix} \right] = u$. Then
 $\frac{1}{2}\nu_* \in  \tilde{\Delta}_u$. \end{enumerate}
 \end{theorem}
\subsection{}  \label{main} We now turn our attention to the conjugation
action of $G$ on itself, that is we assume that $V = G_{\uni}$ and
$* = 1_G$. Recall the subsets $X^{\wtri}$ ($\wtri \in D_G$) and
$H^{\btri}$ ($\btri \in D_G/G$) introduced in Subsection \ref{Gn}.
\begin{lemma}\label{locally}
Each set $\tilde{H}^{\btri}$ is a closed irreducible variety stable
under the conjugation action of $G$. \end{lemma}
\begin{proof} It is clear that the set $\tilde{H}^{\btri}$ is $G$-stable. To see that it is closed,
consider the set
\begin{equation*}  \Set = \left\{ (gG_0^{\wtri}, x)\,|\,\, \ g^{-1} x g \in G_2^{\wtri}    \right\}
\subset G/G_0^{\wtri} \times \tilde{H}^{\btri}.  \end{equation*} If
we show that $\Set$ is closed, then $\tilde{H}^{\btri}$ is closed
since it is the image under the second projection of a closed set,
and $G/G_0^{\wtri}$ is a complete variety. In fact it is sufficient
to show that $\Set' := \{ (g, x) \ | \ g^{-1} x g \in G_2^{\wtri}$
\} is closed in $G\times G$. Indeed, $\Set$ is isomorphic to the
image of $\Set'$ under the quotient map $\eta\,\colon\,\,G\times
G\rightarrow G/G_0^{\wtri}\times G$ and it is explained in
\cite[p.~67]{St1}, for instance, that $\eta$ maps closed subsets of
$G\times G$ consisting of complete cosets of $G_0^{\wtri}\times
\{1_G\}$ to closed subsets of $G/G_0^{\wtri}\times G$. The set
$\Set'$ is closed as it is the inverse image of $G_2^{\wtri}$ under
the conjugation morphism $G \times G \rightarrow G$. Finally, the
set $\tilde{H}^{\btri}$ is irreducible since the product map $G
\times G_2^{\wtri} \rightarrow \tilde{H}^{\btri}$ is a surjective
morphism from an irreducible variety. \end{proof}

Next we show that the sets from Subsection~\ref{Gn} defined by
Lusztig are precisely the sets from Subsection~\ref{strata} defined
by Hesselink.
\begin{theorem}\label{uni-pieces} The following are true.
\begin{enumerate} [{\em (i)}]
\item \label{MTsat} The sets $G_2^{\wtri}$ ($\wtri \in D_G$) are the saturated sets of $G_{\uni}$.
\item \label{MTstrata} The sets $H^{\btri}$ ($\btri \in D_G/G$) are the strata of $G_{\uni}$.
\item \label{MTblade} The sets $X^{\wtri}$ ($\wtri \in D_G$) are the blades of $G_{\uni}$.
\end{enumerate}
Furthermore, if $\tilde{\Delta}_G$ denotes the subset of $Y(G)$
consisting of elements which are in some $\tilde{\Delta}_X$, for a
uniformly unstable set $X$, then $\tilde{\Delta}_G =
\frac{1}{2}\tilde{D}_G$.
\end{theorem}
\begin{proof} Let $\wtri \in D_{G}$, and assume that $\mu \in Y(G)$ is associated to
$\wtri$ under the natural map described in Subsection~\ref{Gn}.
Assume that $\omega \in Y(G')$ comes from the same $\BZ$-scheme
theoretic one parameter subgroup of $\mathfrak{T}$ as $\mu$. (Then
$G\mu$ is identified with $G'\omega$ under the canonical bijection
$Y(G)/G \,\leftrightarrow\, Y(G')/G'$.) So there exists
$\tilde{\omega} \in \Hom(\SL_2(\C), G')$ such that
$\tilde{\omega}_*=
 \omega$, as in (\ref{SL2factor}). Let $u' =
\tilde{\omega}\left[
\begin{smallmatrix}
1 & 1 \\
  & 1 \end{smallmatrix} \right] \in G'$.
Then $\frac{1}{2}\omega \in \tilde{\Delta}_{u'}$ by Theorem
\ref{Kraft}(ii).

Recall that $U(\omega)$ is the unipotent radical of $(G')_0^\wtri = L(\omega)$
and let $U_k(\omega)$ have the same meaning as in
Subsection~\ref{KNess}. Let $\bar{u}'$ denote the image of $u'$ in
 $V_2(\omega):=\,U_2(\omega)/U_3(\omega)$.  Recall that
 $V_2(\omega)\,\cong\,\gl'(\omega,2)$ as
 $L^\perp(\omega)$-modules; see Remark~\ref{overZ}.
 By Theorem \ref{KNuni} the vector $\bar{u}'$ is $L^{\perp}(\omega)$-semistable.
 Since $V_2(\omega)\cong \gl'(\omega,2)$ and the action on it by $L^{\perp}(\omega)$ are
 defined over $\Z$ there exists an affine scheme
 $\mathcal{V}_2(\omega)_{ss}$, acted on by $\mathfrak{L}^{\perp}$, such that
 $\mathcal{V}_2(\omega)_{ss}(\C)\, = \,V_2'(\omega)_{ss}$.
(One should keep in mind here that
$L^{\perp}(\omega)=\,\mathfrak{L}^\perp(\k)$
 thanks to Proposition~\ref{Zscheme}.)
 Since
 $\bar{u}' \in V_2(\omega)$, applying Theorem~\ref{Seshadri} shows that
 $\mathcal{V}_2(\omega)_{ss}$
 has content over any algebraically closed field. So over $\k$,
 there exists $\bar{u} \in V_2(\mu)\,\cong\,\gl(\mu,2)$ which is
 $L^{\perp}(\mu)$-semistable. Let $u$ be a preimage of $\bar{u}$ in
 $ U_2(\mu)$.
 By applying Theorem~\ref{KNuni} again we see that $\mu$ is optimal
 for $u$. Also, since $\frac{1}{2}\omega \in \tilde{\Delta}_{u'}$,
 we see that $\frac{1}{2}\mu \in \tilde{\Delta}_{u}$. Hence $G^{\wtri}_2\,
 =\,U_2(\mu)$ is a saturated set.

Conversely, suppose that $S$ is a non-trivial saturated set in
$G_{\uni}$. We may assume that $S= S(u)$ for some unipotent element
$u\ne 1_G$; see Lemma~\ref{saturated}(\ref{satii}), for example. Let
$\lambda \in \Delta_u$ and $k=m(\lambda,u)$. Then $S =
U_k(\lambda)$. Replacing $u$ by a $G$-conjugate we may assume
further that $\lambda\in Y(T)$. As before, we identify $Y(T)$ and
$Y(T')$. Let $\bar{u}$ denote the image of $u$ in
$V_k(\lambda)=\,U_k(\lambda)/U_{k+1}(\lambda)$.  Theorem \ref{KNuni}
then implies that $\bar{u} \in V_k(\lambda)$ is
$L^{\perp}(\lambda)$-semistable. Since
$V_k(\lambda)\cong\,\gl(\lambda,k)$ as $L^\perp(\lambda)$-modules by
Remark~\ref{overZ}, we may again obtain an affine scheme
$\mathcal{V}_k(\lambda)_{ss}$, defined over $\BZ$ and acted on by
$\mathfrak{L}^{\perp}$, such that
$\mathcal{V}_k(\lambda)_{ss}(\k) = \,V_k(\lambda)$. Applying
Theorem~\ref{Seshadri} we again see that
$\mathcal{V}_k(\lambda)_{ss}$ has content over any algebraically
closed field, and may therefore find $e'
\in\gl'(\lambda,k)_{ss}\cong\, \mathcal{V}_k(\lambda)_{ss}(\C)$; see
Remark~\ref{overZ}.

By applying Theorem \ref{KNuni} we see that $\lambda$
is optimal and primitive for $e'$. Since we are now in
characteristic zero, the Jacobson--Morozov theorem yields that there
exist $f', h' \in \gl'$ such that $(e',h',f')$ is an
$\s\l_2$-triple. Now let $\lambda' \in \Hom(\SL_2(\C),G')$ be such
that $\lambda'_* \in Y(G')$ is adapted to $e'$, so that
$e'\in\gl'(\lambda'_*,2)$. Applying Theorem \ref{Kraft} we see that
$\frac{1}{2}\lambda'_* \in \tilde{\Delta}_{e'}$. Hence
$P(\frac{1}{2}\lambda'_*)\, =\, P(\lambda) \,=\, P(e')$. Since all
maximal tori in $P(e')=\,L(\lambda)\cdot R_u(P(e'))$ are conjugate
we can find $g \in R_u(P(e'))$ such that $\im(\lambda'_*)$ and
$g (\im\lambda) g^{-1}$ lie in the same maximal torus, $T$
say. Note that $g\cdot\lambda$ is optimal for $(\Ad g)\, e'\in
e'+\sum_{i>k}\gl'(\lambda,i)$. Applying Lemma~\ref{tsu} we see that
$g\cdot\lambda$ is optimal for $e'$ as well. Then $g\cdot\lambda
\in\BQ^\times\lambda'_*$ by Theorem \ref{KRthm}(iii). It is well-known
that $\lambda'_* \in \tilde{D}_{G'}$ (see, e.g.,
\cite[Proposition 5.5.6]{Car2}), hence $g^{-1}\cdot\lambda'_* \in
\tilde{D}_{G'}$. But $g^{-1}\cdot\lambda'_* = \lambda$ if
$\lambda'_*$ is primitive and $g^{-1}\cdot\lambda'_* = 2\lambda$
otherwise. So we conclude that $\frac{2}{k}\lambda \in
\tilde{D}_{G'}$ in all cases.  Then, associating a suitable $\wtri
\in D_G$ to $\frac{2}{k}\lambda$, we have that $S =
\,U_2(\frac{2}{k}\lambda) =\, G^{\wtri}_2$. This completes the proof
of (\ref{MTsat}). The claim that $\tilde{\Delta}_G =
\,\frac{1}{2}\tilde{D}_G$ also easily follows from these arguments.
Part~(\ref{MTstrata}) now follows from (\ref{MTsat}) and Proposition
\ref{stratasetminus}. Part~(\ref{MTblade}) then follows from
(\ref{MTsat}) and (\ref{MTstrata}). \end{proof}

\subsection{} We are now in a position to prove one of our main results.

\begin{theorem}

Properties $\gP_1$--\,\,$\gP_4$ hold for any connected reductive
group over any algebraically closed field. Moreover,  $C_G(u)\subset
G_\wtri^0$ for any $u\in X^\wtri$.\end{theorem}

\begin{proof} Properties $\gP_1$ and $\gP_3$ are immediate by Theorem \ref{main}
since the blades and strata are equivalence classes on $G_{\uni}$.
That the sets $X^{\wtri}$ ($\wtri \in \btri$) form a partition of
$H^{\btri}$ for any $\btri \in D_G/G$ is also clear since $H^{\btri}
=\, \bigsqcup_{\wtri \in \btri} X^{\wtri}$. Let $g \in G^{\wtri}_3$
and $u \in X^{\wtri}$. Clearly $gu \in G^{\wtri}_2$. Let $\lambda
\in \Delta_u$ and let $u_k$ be the minimal component of $u$ with
respect to $\lambda$. By the commutator relations $u_k$ is also the
minimal component of $gu$ with respect to $\lambda$. By Theorem
\ref{KNuni} we see that $\Delta_u = \Delta_{gu}$. Now $\| u \|$, $\|
gu \|$ are determined by the minimal component with respect to (any)
optimal one parameter subgroup. Hence, $\| u \| = \| gu \|$ by Lemma
\ref{saturated}(\ref{satii}), and so $gu \in H^{\btri}$ by
Proposition \ref{bladestrataprop}(\ref{bsiv}) and Theorem
\ref{main}. Hence $G^{\wtri}_3X^{\wtri} = X^{\wtri}$. Similarly
$X^{\wtri}G^{\wtri}_3 = X^{\wtri}$, and so $\gP_4$ holds for $G$.
Since the parabolic subgroup $G_\wtri^0=P(\lambda)$ is optimal for
$u$, Theorem~\ref{KRthm}(iv) implies that $C_G(u)\subset G_\wtri^0$.
\end{proof}

\section{Admissible modules and the Hesselink
stratification}\label{admiss}
\subsection{} \label{modulestrata} Previously we did
not restrict $\chara \k$ but for this section and the next it will
be convenient to assume that $\chara \k = p > 0$. As in
Subsection~\ref{poly} we denote by $\mathfrak{G}$ a reductive
$\BZ$-group scheme split over $\BZ$ and write
$G'=\mathfrak{G}(\mathbb{C})$ and $G=\mathfrak{G}(\k)$. Then $G'$
and $G$ are connected reductive groups over $\mathbb{C}$ and $\k$
respectively. Let $V'$ be a finite-dimensional rational $G'$-module.
Given an admissible lattice $V'_\BZ$ in $V'$ we set $V:= V'_{\BZ}
\otimes_\BZ \k$. We call $V$ an {\it admissible $G$-module}. Since
the lattice $V'_\BZ$ is stable under the action of the
distribution $\BZ$-algebra ${\rm Dist}(\mathfrak{G})$, the
$\k$-vector space $V$ is a module over ${\rm Dist}(G)={\rm
Dist}(\mathfrak{G})\otimes_\BZ\k$. This gives $V$ a rational
$G$-module structure; see \cite[\S II.1]{Jan} for more details.

Let $\mathfrak{T}$ be a toral group subscheme of $\mathfrak{G}$ such
that $T':=\mathfrak{T}(\C)$ is a maximal torus of $G'$ and
$T:=\mathfrak{T}(\k)$ is a maximal torus of $G$. We may and will
identify the groups of rational characters $X(T')$ and $X(T)$ and
their duals $Y(T')$ and $Y(T)$. The lattice $V'_\BZ$ decomposes over
$\BZ$ into a direct sum $V'_\Z=\bigoplus_{\mu\in
X(T)}\,V'_{\BZ,\mu}$ of common eigenspaces for the action of
distribution algebra ${\rm Dist}(\mathfrak{T})\subset {\rm
Dist}(\mathfrak{G})$ and base-changing this direct sum decomposition
we obtain the weight space decompositions $V'=\bigoplus_{\mu\in
X(T)} V'_\mu$ and $V=\bigoplus_{\mu\in X(T)}\,V_\mu$ of $V'$ and $V$
with respect to $T'$ and $T$ respectively; see \cite[II1.1(2)]{Jan}.
We mention for completeness that $\dim_{\mathbb C} V'_\mu=\dim_\k
V_\mu$ for all $\mu\in X(T)$.

\begin{theorem}\label{hesselinkstrata}

The following are true.

\begin{enumerate}  [{\em (i)}]
\item \label{mod1} Let $\mathcal{S}'$ and $\mathcal{S}$ denote the
collections of saturated sets of $V'$ and $V$ associated with the
one parameter subgroups in $Y(T')$ and $Y(T)$ respectively. There
exists a collection $\mathfrak S$ of
$\mathrm{Dist}(\mathfrak{T})$-stable direct summands of $V'_{\Z}$
such that
$$\mathcal{S}'=\{S\otimes_\BZ\mathbb{C}\,|\,\,S\in\mathfrak{S}\}\,\mbox{
and } \,\mathcal{S}=\{S\otimes_\BZ \k\,|\,\,S\in\mathfrak{S}\}.$$

\item \label{mod1'} For every $S\in\mathfrak{S}$ we have that
$\Delta(S\otimes_\BZ\mathbb{C})\cap Y_{\mathbb
Q}(T')=\,\Delta(S\otimes_\BZ\k)\cap Y_{\mathbb Q}(T)$.

\item \label{mod2} The strata of $V$ are parametrised by those of $V'$.

\item \label{mod3} The parametrisation from (iii) respects the dimensions of the strata.
In particular, the dimensions of the nullcones of $V'$ and $V$
agree.
\end{enumerate}
\end{theorem}
\begin{proof}
(\ref{mod1}) Let $v' \in V'$ and $v\in V$ be unstable relative
to $T'$ and $T$ respectively. Let $\lambda'$ and $\lambda$ be the
sole elements of $\tilde{\Delta}_{v',T'}$ and
$\tilde{\Delta}_{v,T}$ respectively. Then $S(v') = \bigoplus_{\langle \mu,
\lambda'\rangle \ge 1} V'_{\mu}$ and $S(v) = \bigoplus_{\langle \mu,
\lambda\rangle \ge 1} V_{\mu}.$ As we mentioned earlier, for every
$\mu\in X(T)$ we have that $V_{\mu}'=V_{\mu,
\BZ}\otimes_\BZ\mathbb{C}$ and $V_{\mu}=V_{\mu, \BZ}\otimes_\BZ\k$.
Since the sets of weights of $V'$ and $V$ in $X(T')=X(T)$ coincide,
part~(i) follows.

(\ref{mod1'}) Let $S\in\mathfrak{S}$. Our proof of part~(i) and
Remark~\ref{Hessopt} then show that $S=V'(\lambda)_k\cap V'_\BZ$ for
some $\lambda\in Y(T')=Y(T)$ and some positive integer $k$. Put
$\mathfrak{L}^\perp=\mathfrak{L}^\perp(\lambda)$ and consider the
actions of $\mathfrak{L}^\perp(\mathbb{C})$ and
$\mathfrak{L}^\perp(\k)$ on $V'(k,\lambda)$ and $V(k,\lambda)$
respectively. By Theorem \ref{Seshadri}, there is an open subscheme
$\mathcal{V}(\lambda,k)_{ss}$ of $V_\BZ(\lambda,k) :=
V'(k,\lambda)\cap V'_\BZ$ with the property that
$\mathcal{V}(\lambda,k)_{ss}(\C)$ is the set of
$\mathfrak{L}^{\perp}(\mathbb{C})$-semistable vectors of
$V'(k,\lambda)$ and $\mathcal{V}(\lambda,k)_{ss}(\k)$ is the set of
$\mathfrak{L}^\perp(\k)$-semistable vectors of $V(k,\lambda)$. On
the other hand, Theorem~\ref{Tsujii} tells us that $\lambda$ is
optimal for an element in $V'(\lambda)_k$ (resp. in $V(\lambda)_k$)
if and only if $\mathcal{V}(\lambda,k)_{ss}(\mathbb{C})\ne \emptyset
$ (resp. $\mathcal{V}(\lambda,k)_{ss}(\k)\ne \emptyset$). This
shows that either both sets $\Delta(S\otimes_\BZ\mathbb{C})\cap
Y_{\mathbb Q}(T)$ and $\Delta(S\otimes_\BZ\k)\cap Y_{\mathbb Q}(T)$
are empty or there exists a natural number $m=m(S)$ such that
$$\Delta(S\otimes_\BZ\mathbb{C})\cap Y_{\mathbb
Q}(T)\,=\,\Delta(S\otimes_\BZ\k)\cap Y_{\mathbb
Q}(T)=\frac{1}{m}\lambda.$$ This proves part~(ii).

(\ref{mod2}) Consider a stratum $G'[v]\subset V'$. Without loss of
generality we may assume that the blade $[v]$ is $T'$-unstable,
since all maximal tori are conjugate in $G'$. Then part~(ii) gives
us a blade $[w]\subset V$ corresponding to $[v]$. Since all maximal
tori in $G$ are conjugate as well, part~(ii), in conjunction with our
discussion in Subsection~\ref{strata}, shows that any stratum
$G[w]\subset V$ is obtained by the above construction in a unique
way. Then the map $G'[v] \mapsto G[w]$ defines the required
parametrisation.

(\ref{mod3}) With $[v] \subset V'$ and $[w] \subset V$ as above we
have that $$\dim_\mathbb{C} G'[v]\, = \,\dim_{\mathbb{C}} G' -
\dim_{\mathbb C} P(v) + \dim_{\mathbb C} S(v)$$ and $$\dim_\k G[w]
\,=\, \dim_\k G - \dim_\k P(w) + \dim_\k S(w)$$ by
\cite[Proposition 4.5(c)]{Hess2}. By part~(i) we have that
$\dim_{\mathbb C} S(v) = \dim_\k S(w)$, whilst the equality
$\dim_{\mathbb C} P(v) = \dim_\k P(w)$ follows from the definition
of $P(\lambda)$ in Section \ref{lambdasbgrps}. Hence $\dim_{\mathbb
C} G'[v] = \dim_\k G[w]$, as required.

Since the set of $T'$-weights of $V'$ is finite, so is the set
$\{K_T(v')\,|\,\,v'\in V'\}$. Then Lemma~\ref{opti} implies that the
number of $S\in\mathfrak{S}$ with
$\Delta(S\otimes_\BZ\mathbb{C})\cap Y_{\mathbb Q}(T)\ne \emptyset$
is finite, too. In view of our earlier remarks in this part we now
get $\dim_\mathbb{C}\N_{V'}=\dim_\k\N_V$.
\end{proof}

\begin{rmk} 1. In general, different lattices $V'_{\Z}$ may give rise to
non-isomorphic $G$-modules. On the other hand, the theorem implies
that the stratification does {\em not} depend on the choice of
lattice and remains essentially the same over any algebraically
closed field.

\smallskip

\noindent 2. Let $E(\lambda)$ denote the finite dimensional
irreducible $G$-module with highest weight $\lambda\in X(T)$. Then
it is well-known that $\lambda$ is a dominant weight and there
exists an admissible lattice, $V''_\Z(\lambda)$, in the irreducible
finite dimensional $\gl'$-module $V'(\lambda)$ with highest weight
$\lambda$ such that $E(\lambda)$ is isomorphic to a submodule of the
$G$-module $V''_\k(\lambda):=\,V''_\Z(\lambda)\otimes_\Z\k$; see
\cite[\S12, Exercise after Theorem~39]{St}. If $\nu\in Y(G)$ is
optimal for a non-zero $G$-unstable vector $v\in E(\lambda)$, then
the definition in Subsection~\ref{KRthm} shows that it remains so
for $v$ regarded as a vector of $V''_\k(\lambda)$. Therefore the
Hesselink strata of $E(\lambda)$ are precisely the intersections of
those of $V''_\k(\lambda)$ with $E(\lambda)$. Now
Theorem~\ref{hesselinkstrata}(iii) implies the Hesselink strata of
$E(\lambda)$  are parametrised by a subset of the Hesselink strata
of the $\gl'$-module $V'(\lambda)$.

\end{rmk}

\subsection{}\label{finitefield} In this subsection we assume that $\k$
is an algebraic closure of $\F_p$. Keeping the notation of
Subsection~\ref{rootdatahom} we assume that $(X(T), \Sigma, Y(T),
\Sigma^\vee)$ is the root datum of the reductive group scheme
$\mathfrak{G}$. Let $G=\mathfrak{G}(\k)$ and write $x_\alpha(t)$
for Steinberg's generators of the unipotent root subgroups
$U_\alpha$ of $G$; see \cite{St}. Choose a basis of simple roots $\Pi$
in $\Sigma$ and denote by $Y^+(T)$ the Weyl chamber in $Y(T)$ associated
with $\Pi$. (It consists of all $\mu\in Y(T)$ such that
$\langle\alpha,\mu\rangle\ge 0$ for all $\alpha\in \Pi$.) Let $\tau$
be an automorphism of the lattice $X(T)$ and denote by $\tau^*$ the
natural action of $\tau$ on $Y(T)=\,\Hom_\BZ(X(T),\BZ)$. Assume
further that $\tau$ preserves both $\Sigma$ and $\Pi$ and $\tau^*$
preserves $\Sigma^\vee$. Finally, assume that the quadratic form $q$ from
Subsection~\ref{KRthm} is invariant under $\tau^*$.

Now fix a $p^{th}$ power $q=p^l$. Then it is well-known that $\tau$
gives rise to a Frobenius endomorphism $F=F(\tau, l)\colon\,x\mapsto
x^F$, of the algebraic $\k$-group $G=\mathfrak{G}(\k)$. The
endomorphism $F$ is uniquely determined by the following properties:
\begin{enumerate}
\item[1.] \ $(\tau\eta)(x^F)=\eta(x)^q\ $ for all $\eta\in X(T)$ and $x\in T$;

\smallskip

\item[2.] \ $\lambda(t)^F=(\tau^*\lambda)(t^q)\ $ for all
$\lambda\in Y(T)$ and $t\in\k^\times$;

\smallskip

\item[3.]\ $x_\alpha(t)^F=x_{\tau\alpha}(t^q)\ $ for all $\alpha\in
R$ and $t\in \k$;
\end{enumerate}
see \cite[Theorem~3.17]{DM} for instance.
Let $V$ be an admissible $G$-module endowed with an action of $F$
such that
\begin{equation}\label{frob}
g(v)^F=g^F(v^F)\quad\ \mbox {for all }\,  g\in G\, \mbox { and }\,
v\in V.
\end{equation}
As usual we require that the action of $F$ is {\it $q$-linear}, that
is $(\lambda v)^F =\lambda^qv^F$ for all $\lambda\in \k$ and $v\in
V$, and that each vector in $V$ is fixed by a sufficiently large
power of $F$. In this situation one knows that the fixed point space
$V^F$ is an $\F_q$-form of $V$. In particular,
$\dim_{\F_q}V^F=\,\dim_\k V$; see \cite[Corollary~3.5]{DM}. We
mention, for use later, that there is a natural $q$-linear action of
$F$ on the dual space $V^*$, compatible with that of $G$ (recall
that $G$ acts on $V^*$ via $(g\cdot\xi)(v)=\xi(g^{-1}\cdot v)$ for
all $g\in\ G$, $\xi\in V^*$, $v\in V$). Since $V^F$ is an
$\F_q$-form of $V$, the dual space $(V^F)^*$ contains a $\k$-basis
of $V^*$, say $\xi_1,\ldots, \xi_m$. Then every $\xi\in V^*$ can be
uniquely expressed as a linear combination $\xi=\,\sum_{i=1}^m
\lambda_i\xi_i$ with $\lambda_i\in \k$ and we can define
$F\colon\,V^*\rightarrow V^*$ by setting
$\xi^F:=\,\sum_{i=1}^m\lambda_i^q\xi_i$. Verifying (\ref{frob}) for
this action of $F$ reduces to showing that $g^{-1}g^F(\xi)=\xi$ for
all $\xi\in (V^F)^*$ and $g\in G$, which is clear because
$(g^{-1})^Fg(v)=v$ for all $v\in V^F$.

There are many reasons to be interested in the cardinality of the
finite set ${\N_V}^F=\N_V\cap V^F$, and here we can offer the
following general result.
\begin{theorem}\label{NFq}
Under the above assumptions on $F$ and $V$ there exists a polynomial
$n_V(t)\in \BZ[t]$ such that $|{\N_V}^F|=\,n_V(q)$ for all $q=p^l$.
The polynomial $n_V(t)$ depends only on $V'$ and $\tau$, but not on
the choice of an admissible lattice $V'_\BZ$, and is the same for
all primes $p\in\mathbb{N}$.
\end{theorem}
\begin{proof}
Let $\Lambda(V)$ denote the set of pairs $(\lambda,k)$ where
$\lambda\in Y^+(T)$ is primitive and $k$ is a positive integer such
$\mathcal{V}(\lambda,k)_{ss}(\k)\ne \emptyset$ (the notation of
Subsection~\ref{modulestrata}). Set
$\Lambda(V,\tau)=\{(\lambda,k)\in\Lambda(V)\,|\,\,\tau^*\lambda=\lambda\}$
and define
$$\mathcal{H}(\lambda,k)\,:=\,
G\cdot\Big(\mathcal{V}(\lambda,k)_{ss}(\k)\oplus\textstyle{\bigoplus}_{i>k}V(\lambda,i)\Big),$$
the Hesselink stratum associated with $(\lambda,k)\in\Lambda(V)$.
Recall that $\mathcal{V}(\lambda,k)_{ss}(\k)=V(\lambda,k)\setminus
\N_{V(\lambda,k)}$ where $\N_{V(\lambda,k)}$ is the set of all
$L^\perp(\lambda)$-unstable vectors of $V(\lambda,k)$. To ease
notation we set $$V(\lambda,\,\ge
k)_{ss}:=\,\mathcal{V}(\lambda,k)_{ss}(\k)\oplus\textstyle{\bigoplus}_{i>k}V(\lambda,i).$$

If $\mu\in Y(G)$ is optimal for a non-zero vector $v\in {\N_V}^F$,
then so is $\mu^F$, forcing $P(v)=P(\mu)=P(\mu^F)=P(v)^F$. So the
optimal parabolic subgroup of $v$ is $F$-stable. But then $P(v)$
contains an $F$-stable Borel subgroup which, in turn, contains an
$F$-stable maximal torus of $G$; we shall call it $T_1$. Since both $T$
and $T_1$ are $F$-stable maximal tori contained in $F$-stable Borel
subgroups of $G$, there is an element $g_1\in G^F$ such that
$T_1=g_1^{-1}Tg_1$; see \cite[3.15]{DM}. Then $Y(T)$ contains an
optimal one parameter subgroup for $g_1(v)\in V^F$, say $\mu_1$.
Lemma~\ref{opti}(iv) yields $\tau^*\mu_1=\mu_1$. Since the unipotent
radical $U(\mu_1)$ of $P(\mu_1)$ is contained in the Borel subgroup
of $G$ associated with our basis of simple roots $\Pi$, we see that
$\mu_1\in Y^+(T)$.

Now suppose $v\in \mathcal{H}(\lambda,k)^F$, so that $v=gw$ for
some $w\in V(\lambda,\,\ge k)_{ss}$ and $g\in G$ . Let $g_1\in G^F$
and $\mu_1\in Y^+(T)$ be as above (so that $\mu_1$ is optimal for
$v_1=g_1(g'w)\in V^F$). Note that $T\subset L(\mu_1)\subset P(v_1)$.
We may assume without loss of generality that $\mu_1$ is primitive
in $Y(G)$. Since $w$ and $v_1$ are in the same Hesselink stratum of
$V$ it must be that $G\cdot \Delta_{v_1}=G\cdot \Delta_{w}$. This
yields the equality $(G\cdot\mu_1)\cap Y(T)\,=\, (G\cdot\lambda)\cap
Y(T)$ which, in turn, implies that that $\mu_1$ and $\lambda$ are
conjugate under the action of the Weyl group $W$ on $Y(T)$. Since
both $\lambda$ and $\mu_1$ are in $Y^+(T)$, we get $\mu_1=\lambda$.

As a result, we deduce that $\tau^*\lambda=\lambda$. Hence both
$P(\lambda)$ and ${V(\lambda,\,\ge k)_{ss}}$ are $F$-stable.
Applying \cite[Proposition~4.5(b)]{Hess2} now yields that  $g^F\in
gP(w)$. We choose in $G^F$ a set of representatives
$\mathcal{X}(\lambda,\tau, q)$ for $G^F/P(\lambda)^F$, so that
$$|\mathcal{X}(\lambda,\tau, q)|\,=\,|G^F/P(\lambda)^F|.$$ As
$P(\lambda)$ is an $F$-stable connected group, the Lang--Steinberg
theorem shows that $g^{-1}g^F=x^{-1}x^F$ for some $x\in P(v)$; see
\cite[Theorem 3.10]{DM} for instance. Then $gx^{-1}\in P(\lambda)^F$
and hence no generality will be lost by assuming that
$g\in\mathcal{X}(\lambda,\tau, q)$.

According to \cite[Proposition~4.5(b)]{Hess2} there is an
$F$-equivariant bijection between the fibre product
$G\times^{P(\lambda)}V(\lambda,\,\ge
k)_{ss}\cong\,(G/P(\lambda))\times V(\lambda,\,\ge k)_{ss}$ and the
stratum $\mathcal{H}(\lambda,k)$. Since $v\in V^F$ and $g\in G^F$ we
have that $g (w^F)=g w$, which shows that $w\in {V(\lambda,\,\ge
k)_{ss}}^F$. As a consequence,
\begin{equation}\label{HFq}
|\mathcal{H}(\lambda,k)^F|=|\mathcal{X}(\lambda,\tau, q)|\cdot
|{V(\lambda,\,\ge k)_{ss}}^F|=f_{\tau,\lambda}(q)\cdot
q^{N(\lambda,\,k)}
\Big(q^{n(\lambda,k)}-|{\N_{V(\lambda,k)}}^F|\Big)\end{equation}
where $f_{\tau,\lambda}(q)=|\mathcal{X}(\lambda,\tau, q)| =
|G^F/P(\lambda)^F|$,\ $N(\lambda,k)=\sum_{i>k}\, \dim V(\lambda,i)$,
and $n(\lambda,k)=\dim V(\lambda,k)$.

After these preliminary remarks we are going to prove our theorem by
induction on the rank of $G$. If ${\rm rk}\,G=0$, then $G=\{1_G\}$
and hence $\k[V]^G=\k[V]$. Therefore ${\N_V}^F=\{0\}$ and we can
take $1$, a constant polynomial, as $n_V(t)$. Now suppose that ${\rm
rk}\,G>0$ and our theorem holds for all connected reductive groups
of rank $<{\rm rk}\,G$. Since for every $\lambda\in \Lambda(V,\tau)$
we have that ${\rm rk}\,L^\perp(\lambda)<{\rm rk}\,G$ and each
$L^\perp(\lambda)$-module $V(\lambda,i)$ is admissible by our
discussion in Subsection~\ref{modulestrata}, there exist polynomials
$n_{V(\lambda,i)}(t)\in \BZ[t]$ with coefficients independent of $p$
and our choice of an admissible lattice $V_\BZ'(\lambda,i)$ in
$V'(\lambda,i)$ such that
$|{\N_{V(\lambda,i)}}^F|=\,n_{V(\lambda,i)}(q)$.

Next we note that for every $\lambda\in Y(T)$ with
$\tau^*\lambda=\lambda$ there is a polynomial
$f_{\tau,\lambda}\in\BZ[t]$ with coefficients independent of $p$
such that $f_{\tau,\lambda}(q)=|G^F/P(\lambda)^F|$ for all $p^{th}$
powers $q$ and all $p$. Indeed, it is immediate from
\cite[Proposition~3.19(ii)]{DM} that $f_{\tau,\lambda}$ can be
chosen as a quotient $a_{\tau,\lambda}/b_{\tau,\lambda}$ of two
coprime polynomials $a_{\tau,\lambda}, b_{\tau,\lambda}\in\BZ[t]$
with coefficients independent of $p$. Since
$f_{\tau,\lambda}(q)\in\BZ$ for infinitely many $q\in\Z$, it must be
that $\deg\,b_{\tau,\lambda}=0$. Therefore $f_{\tau,\lambda}\in
\mathbb{Q}[t]$. On the other hand, $G^F/P^F$ is the set of
$\F_q$-rational points a smooth projective variety defined over
$\F_p$. Applying \cite[Lemma~2.12]{GR} one obtains that
$f_{\tau,\lambda}\in \BZ[t]$, as stated.

Putting everything together we now get
\begin{eqnarray*}
|{\N_V}^F|&=&1+\sum_{(\lambda,\,k)\in\,\Lambda(V,\tau)}|\mathcal{H}(\lambda,k)^F|\\
&=&1+\sum_{(\lambda,\,k)\in\,\Lambda(V,\tau)}f_{\lambda,\tau}(q)\cdot
q^{N(\lambda,k)}\Big(q^{n(\lambda,k)}-n_{V(\lambda,k)}(q)\Big).
\end{eqnarray*}
Since the data $\big\{\big(n(\lambda,k),
N(\lambda,k)\big)\,|\,\,(\lambda,k)\in \Lambda(V,\tau)\big\}$ arrives
unchanged from the $G'$-module $V'$ and is independent of $p$ by
Theorem~\ref{hesselinkstrata}, the RHS is a polynomial in $q$ with
integer coefficients independent of $p$ and the choice of
admissible lattice in $V'$.
\end{proof}
\begin{rmk} In the notation of Subsection~\ref{modulestrata},
the distribution algebra ${\rm Dist}_\BZ(\mathfrak{G})$ acts
naturally on the $\BZ$-algebra $\BZ[V'_\BZ]$ and we may consider the
invariant algebra of this action, which coincides with
$\BZ[V'_\BZ]^{\mathfrak{G}}$. According to \cite[\S II]{Sesh}, the
algebra $\BZ[V'_\BZ]^{\mathfrak{G}}$ is generated over $\BZ$ by
finitely many homogeneous elements. The ideal of $\BZ[V'_\BZ]$
generated by these elements defines a closed subscheme of the affine
scheme ${\rm Spec}\,\BZ[V'_\BZ]$ which we denote by
$\mathcal{N}(V_\BZ')$. It follows from \cite[Proposition~6(2)]{Sesh}
that for any prime $p\in\BN$ the nullcone $\N_V$ coincides with the
variety of closed points of the affine $\k$-scheme
$\N(V_\BZ')\times_{{\rm Spec}\,\Z}\,{\rm Spec}\,\k$. At this point
Theorem~\ref{NFq} shows that the affine $\BZ$-scheme $\N(V_\BZ')$ is
{\it strongly polynomial-count} in the terminology of N.~Katz.
Applying \cite[Theorem~1(3)]{Katz} we now deduce that the polynomial
$n_V(t)$ from Theorem~\ref{NFq} is closely related with the {\it
$E$-polynomial}
$E(\N_{V'};x,y)\,=\,\sum_{i,j}e_{i,j}x^iy^j\in\BZ[x,y]$ of the
complex algebraic variety $\N_{V'}$.  More precisely, we have that
$E(\N_{V'};x,y)=n_V(xy)$ as polynomials in $x,y$; see
\cite[p.~618]{Katz} for more details. This shows that the
coefficients of $n_V(t)$ are determined by Deligne's mixed Hodge
structure on the compact cohomology groups
$H_c^k(\N_{V'},\mathbb{Q})$.
\end{rmk}

Define $n'_V(t):=(n_V(t)-1)/(t-1)$. As $n'_V(q)={\rm
Card}\,\big\{\F_q^\times\, v\,|\,\,v\in {\N_V}^F,\ v\ne 0\big\}$ for
all $p^{th}$ powers $q$, it is straightforward to see that $n'_V(t)$
is a polynomial in $t$. The long division algorithm then shows that
$n'_V(t)\in \Z[t]$. We conjecture that the polynomial $n'_V(t)$ has
{\it non-negative} coefficients. This conjecture holds true for
$\mathfrak{G}= \mathbf{SL}_2$ where one can compute $n'_V(t)$
explicitly for any admissible $G$-module $V$. The details are left
as an exercise for the interested reader.
\section{Nilpotent pieces in $\gl$ and $\gl^*$}\label{dual}
\subsection{}\label{morepieces} We now define nilpotent pieces in the Lie algebra $\gl$ completely
analogously to the definition of unipotent pieces, that is, we
partition $\gl_{\nil}=\N_\gl$ into smooth locally closed $G$-stable
pieces, indexed by the unipotent classes in $G'=\mathfrak{G}(\C)$.
For convenience, we now allow $\chara \k = p \ge 0$.  For $\wtri
\in D_G$ and $i \ge 0$ we define $\gl_{i}^{\wtri} = \Lie
G_i^{\wtri}$. For any $G$-orbit $\btri \in D_G$, let
$\tilde{H}^{\btri}(\gl) = \bigcup_{\wtri \in \btri}
\gl_{2}^{\wtri}$. This is a closed irreducible $G$-stable variety by
the proof of Lemma \ref{Gn}. We define the {\em nilpotent pieces} of
$\gl$ to be the sets
\begin{equation*}  {H}^{\btri}(\gl)  := \,\tilde{H}^{\btri}(\gl) \setminus
\textstyle{\bigcup}_{\btri'}\, \tilde{H}^{\btri'}(\gl),
\end{equation*}
where the union is taken over all $\btri' \in D_G/G$ such that
$\tilde{H}^{\btri'}(\gl) \subsetneqq \tilde{H}^{\btri}(\gl)$. We
also define
\begin{equation*}  {X}^{\wtri}(\gl)  := \, \gl_{2}^{\wtri}\, \textstyle{\bigcap}\, H^{\btri}(\gl),
\end{equation*}
for each $\wtri \in D_G$, where $\btri$ is the $G$-orbit of $\wtri$.
Since ${H}^{\btri}(\gl)$ is the complement of finitely many
non-trivial closed subvarieties of $\tilde{H}^{\btri}(\gl)$, it is
open and dense in $\tilde{H}^{\btri}(\gl)$, hence it is locally
closed in $\gl_{\nil}$. The subset ${H}^{\btri}(\gl)$ is $G$-stable
since its complement in $\tilde{H}^{\btri}(\gl)$ is. Consequently,
${X}^{\wtri}(\gl)$ is open and dense in $\gl_{2}^{\wtri}$, and
stable under the adjoint action of $G_0^{\wtri}$.

 Recall from Subsections~\ref{Kraft} and
~\ref{main} that for any $\wtri\in D_G$ there is an element $g\in G$
and a one parameter subgroup $\omega\in Y(T)=Y(T')$, coming from a
rational homomorphism $\SL_2(\C)\rightarrow G'$, such that
$\frac{1}{2}\omega\in\tilde{\Delta}_x$ for some $x\in
{\gl'}(2,\omega)$  and $\gl_{k}^\wtri\,=\,\bigoplus_{i\ge
k}\,\gl(i,g\cdot\omega)$ for all $k\in\BZ$. Note that different
$g\in G$ with this property have the same image in
$G_0^{\wtri}\backslash G$. Given $\mu\in Y(G)$ and $i\in\BZ$ we
denote by $\gl^*(i, \mu)$ the subspace in $\gl^*$ consisting of all
linear functions that vanish on each $\gl(j,\mu)$ with $j\ne -i$.
Now define $(\gl^*)^\wtri_k\,:=\,\bigoplus_{i\ge k}\,\gl^*(i,
g\cdot\omega)$, for $k\in\BZ$. The preceding remark shows that this
is independent of the choice of $g\in G$ and therefore the subspaces
$(\gl^*)^\wtri_k$ are well-defined.

In a completely analogous way we now define the {\it nilpotent
pieces} of the dual space $\gl^*$. For any $G$-orbit $\btri \in
D_G$, we let $\tilde{H}^{\btri}(\gl^*) = \bigcup_{\wtri \in \btri}
(\gl^*)_{2}^{\wtri}$, a closed irreducible $G$-stable subset of
$\gl^*$, and put
\begin{equation*}  {H}^{\btri}(\gl^*)  := \,\tilde{H}^{\btri}(\gl^*) \setminus
\textstyle{\bigcup}_{\btri'}\, \tilde{H}^{\btri'}(\gl^*),
\end{equation*}
where the union is taken over all $\btri' \in D_G/G$ with
$\tilde{H}^{\btri'}(\gl^*) \subsetneqq \tilde{H}^{\btri}(\gl^*)$. We
define
\begin{equation*}  {X}^{\wtri}(\gl^*)  := \, (\gl^*)_{2}^{\wtri}\, \textstyle{\bigcap}\, H^{\btri}(\gl^*),
\end{equation*}
for each $\wtri \in D_G$. Arguing as before we observe that each
${H}^{\btri}(\gl^*)$ is a $G$-stable, locally closed subset of
$\N_{\gl^*}$. Hence ${X}^{\wtri}(\gl^*)$ is open and dense in
$\gl_{2}^{\wtri}$, and stable under the coadjoint action of
$G_0^{\wtri}$.

\subsection{} In the next two subsections we
study the nullcone $\N_{\gl^*}$ associated with the coadjoint action
of $G$ on the dual space $\gl^*=\Hom_{\k\,}(\gl,\k)$. Recall that
$(g\cdot \xi)(x) = \xi( (\Ad g^{-1}) x )$ for all $g \in G$, $x \in
\gl$, $\xi \in \gl^*$. It is immediate from the Hilbert--Mumford
criterion (our Theorem~\ref{HMum}) that $\xi\in\N_{\gl^*}$ if and
only if $\xi$ vanishes on the Lie algebra of a  Borel subgroup of
$G$. The nilpotent linear functions $\xi\in\N_{\gl^*}$ play an
important role in the study of the centre of the enveloping
algebra $U(\gl)$  and were first investigated in our setting by Kac
and Weisfeiler in \cite{KW}. In characteristic zero the Killing form
induces a $G'$-equivariant isomorphism $\gl' \cong (\gl')^*$.
However, in positive characteristic it may happen that
$\gl\not\cong\gl^*$ as $G$-modules.

We first assume that the group $G$ is simple and simply connected.
Rather than study $\gl^*$ directly, we will present a slightly
different construction which will allow us to combine Theorems
\ref{modulestrata} and \ref{NFq} with classical results of Dynkin
\cite{Dyn}  and Kostant \cite{Kos} on $\N_{\gl'}$. As before, we fix
a set of simple roots $\Pi$ in $\Sigma$ and denote the
corresponding set of positive roots by $\Sigma^+$. Let ${\mathcal
C}'\,=\,\{ X_{\alpha}, H_{\beta} \ \vline \ \alpha \in \Sigma,\,
\beta \in \Pi\}$ be a Chevalley basis of $\gl'$ and denote by
$\gl_\Z'$ the $\BZ$-span of ${\mathcal C}'$ in $\gl$. Then the
following equations hold in $\gl'_\BZ$:
\begin{enumerate} [(i)]
\item $[H_{\alpha}, X_{\beta}] = \langle \beta, \alpha \rangle X_{\beta}$ for all $\alpha\in\Pi$,
$\beta\in\Sigma$;
\smallskip
\item $[X_{\beta}, X_{-\beta}] = H_{\beta}$ for all $\beta \in\Pi$, where $H_\beta={\rm d}_e\beta^\vee$
is an integral linear combination of $H_\alpha={\rm d}_e\alpha^\vee$
with $\alpha\in\Pi$;
\smallskip
\item $[X_{\alpha}, X_{\beta}] =  N_{\alpha, \beta}\, X_{\alpha + \beta}$
    if $\alpha + \beta \in \Sigma$, where $N_{\alpha, \beta} = \pm (q+1)$ and $q$ is the
    maximal integer for which $\beta - q \alpha \in \Sigma$;
\smallskip
\item $[X_{\alpha}, X_{\beta}] = 0$ if $\alpha + \beta \notin
    \Sigma$;
\end{enumerate}
see \cite[\S 1]{St}, for example. As usual, $\langle \alpha, \beta
\rangle = 2(\alpha, \beta)/(\alpha, \alpha)$, where $(\ ,\ )$ is a
scalar product on the $\BR$-span of $\Pi$, invariant under the
action of the Weyl group $W$ of $\Sigma$. We may assume, by
rescaling if necessary, that $(\alpha,\alpha ) = 2$ for every short
root $\alpha$  of $\Sigma$. Let $\tilde{\alpha}$ denote the maximal
root, and $\alpha_0$ the maximal short root in $\Sigma^+$
respectively, and set $d := (\tilde{\alpha}, \tilde{\alpha}) /
(\alpha_0, \alpha_0 )$. Recall that a prime $p\in\mathbb N$ is
called {\it special} for $\Sigma$ if $d \equiv 0 \pmod p$. The
special primes are $2$ and $3$. To be precise, $2$ is special for
$\Sigma$ of type $\Btype_\ell$, $\Ctype_\ell$, $\ell\ge 2$, and
$\Ftype_4$, whilst $3$ is special for $\Sigma$ of type $\Gtype_2$.

Since $G$ is assumed to be simply connected, we have that $\gl
=\Lie G = \gl_{\BZ}' \otimes_{\BZ} \k$  (cf. \cite[\S 2.5]{Bor2} or
\cite[\S 1.3]{Jan3}). Also, the distribution algebra ${\rm
Dist}_\BZ(\mathfrak{G})$ identifies canonically with the unital
$\BZ$-subalgebra of the universal enveloping algebra $U(\gl')$
generated by all $X_{\beta}^n/n!$ with $\beta\in\Sigma$ and $n \in
\BN$. The algebra $U_{\BZ}$ is known as {\it Kostant's
$\Z$-form} of $U(\gl)$ and was first introduced in \cite{Kos2}.
Thus, a $\Z$-lattice $V'_\BZ$ in a finite-dimensional $\gl'$-module
$V'$ is admissible if and only if it is invariant under all
operators $X_{\alpha}^n/n!$ ($n \in \BN$) under the obvious action
of $U(\gl')$ on $V'$. For instance, $\gl_{\BZ}'$ itself is
admissible, since $\gl_{\BZ}' =\, U_{\BZ}\cdot X_{\tilde{\alpha}}$.

We now recall very briefly how admissible lattices give rise to
rational $G$-modules. Let $V=V'_\BZ\otimes_\BZ\k$. Since ${\rm
Dist}_\k(G)\,=\,{\rm
Dist}_\BZ(\mathfrak{G})\otimes_{\BZ}\k\,=\,U_\BZ\otimes_\Z\k$, the
action of $U_\BZ$ on $V_\BZ'$ gives rise to a representation of
${\rm Dist}_\k(G)$ on $\End_\k V$, and hence to a rational linear
action of $G$ on $V$; see \cite[\S\S II.1.12, II.1.20]{Jan} for
more details. Given $X \in U_{\Z}$ we denote the induced linear
transformations on $V_{\BZ}'$ and $V$ by $\rho_{\BZ}(X)$. We then
define invertible linear transformations $x_{\beta}(t) = \sum_{n\ge
0} t^n \rho_\BZ(X_{\beta}^n/n!)$ on $V$, for each $\beta \in
\Sigma$, where $t \in \k$. (Note that the sum is finite since the
$X_{\beta}$ act nilpotently on $V'$.) The set $\big\{
x_{\beta}(t)\,\vline \,\, \beta \in \Sigma,\, t \in \k \big\}$
generates a Zariski-closed, connected subgroup $G(V)$ of $\GL(V)$.
Since $G$ is simply connected and hence a universal Chevalley group
in the sense of \cite{St}, the linear group $G(V)$ is a homomorphic
image of $G$. For any admissible lattice $V'_{\BZ}$ in a
finite-dimensional $\gl'$-module $V'$, we thus obtain a $G$-module
structure on $V=V'_{\BZ}\otimes_{\BZ} \k$.

Define a symmetric bilinear form $\langle\,\ , \ \rangle \colon\,
\gl'_{\BZ} \times \gl'_{\BZ} \rightarrow \BZ$ by setting
\begin{eqnarray*}
\langle X_{\alpha}, X_{\beta}
 \rangle  & = & 0 \ \, \ \ \ \ \ \ \ \ \ \ \ \ \ \ \mbox{ if } \alpha + \beta \not= 0, \\
\langle H_{\alpha}, H_{\beta} \rangle  & = & \frac{4d(\alpha,
\beta)}{(\alpha, \alpha)(\beta, \beta)}\ \,
\mbox{ for all } \alpha, \beta \in \Sigma, \\
\langle X_{\alpha}, X_{-\alpha} \rangle  & = & \frac{2d}{(\alpha,
\alpha)} \ \ \ \ \ \ \ \ \  \mbox{ for all } \alpha \in \Sigma,
\end{eqnarray*}
and extending to $\gl_\BZ'$ by $\BZ$-bilinearity. Note that this is
well-defined, since the condition $(\alpha_0, \alpha_0) = 2$ ensures
that the image is indeed in $\BZ$; see Bourbaki's tables in
\cite{Bour}. Obviously we may extend $\langle\,\ , \ \rangle$ to
symmetric bilinear forms $\langle\,\ , \ \rangle_{\C}$ on $\gl' =
\gl_{\BZ}' \otimes_{\BZ} \C$, and $\langle\,\ , \ \rangle_{\k}$ on
$\gl = \gl'_{\BZ} \otimes_{\BZ} \k$.

In is proved in \cite[p.~240]{Pre2} that the bilinear form $\langle\
, \ \rangle_{\C}$ is a scalar multiple of the Killing form $\kappa$
of $\gl'=\Lie G'$. In particular, $\langle\,\ , \ \rangle_{\C}$ is
$G'$-invariant. This, in turn, implies that
\begin{equation}\label{antipode}\langle X(u),v\rangle=\langle
u,X^{\scriptstyle{\top}}(v)\rangle\ \ \,\mbox{for all }\ u,v\in
V'_\BZ \ \mbox{ and }\, X\in U_\BZ,\end{equation} where
$\scriptstyle{\top}$ stands for the canonical anti-automorphism of
$U(\gl)$. Since $x^{\scriptstyle{\top}}=-x$ for all $x\in\gl'$, it
is straightforward to see that $\scriptstyle{\top}$ preserves the
$\BZ$-form $U_\BZ$ of $U(\gl')$. (In fact, the map
${\scriptstyle{\top}}\colon\,U_\BZ\rightarrow U_\BZ$ is nothing but
the antipode of the Hopf algebra $U_\BZ={\rm
Dist}_\BZ(\mathfrak{G})$.) As a consequence, the bilinear form
$\langle\,\ , \ \rangle_{\k}$ on $\gl=\Lie G$ is $G$-invariant.

\begin{lemma} \label{Z-dual} If $p$ is non-special for $\Sigma$,
then the radical of $\langle\,\ , \ \rangle_{\k}$ coincides with
the centre $\mathfrak{z}(\gl)$ of the Lie algebra $\gl$. If $p$ is
special for $\gl$, then $\Rad\, \langle\,\ , \ \rangle_{\k}
\not\subseteq\, \mathfrak{z}(\gl)$.
\end{lemma}
\begin{proof}
The first statement of the lemma is \cite[Lemma 2.2(ii)]{Pre2}. For
the second statement, we note that the image of $X_{\alpha_0}$ in
$\gl\,=\big(\gl'_\BZ/p\gl'_\BZ\big)\otimes_{\F_p}\k$ lies in the
radical of $\langle\,\ , \ \rangle_{\k}$, but not in the centre of
$\gl$. (Recall that $G$ is assumed to be simply connected.)
\end{proof}
The lemma hints at the fact that $\gl$ and $\gl^*$ are similar as
$G$-modules if $p$ is non-special, but very different if $p$ is
special. Nevertheless, as we will see, we may construct an
alternative admissible lattice $\gl''_{\BZ} \subset \gl'$ which
gives rise to another $G$-module $\gl''_{\BZ} \otimes_{\BZ} \k$
such that $\langle\ , \ \rangle$ induces a non-degenerate pairing
between $\gl''_{\BZ} \otimes_{\BZ} \k$ and $\gl$ in all cases. This
will enable us to identify the $G$-modules $\gl''_\BZ\otimes_\BZ\k$
and $\gl^*$.
\subsection{} \label{duallattice} We define
$\gl_{\BZ}'' \,:=\, \{ x \in \gl'\,|\,\, \langle x, y \rangle \in
\BZ, \ \forall\, y \in \gl'_{\BZ}\}$, a $\BZ$-lattice in $\gl'$. It
is immediate from (\ref{antipode}) that $\gl''_\BZ$ is an admissible
lattice. Consequently, we obtain a $G$-module structure on the
vector space $\gl_{\BZ}' \otimes_{\BZ} \k$. We also obtain a
$G$-invariant pairing \begin{equation}\label{form}\langle\ \,, \
\rangle^*_{\k}\,\colon\,\, \gl \times \big(\gl_{\BZ}'' \otimes_{\BZ}
\k\big)\, \longrightarrow\,\, \k.\end{equation} We will now exhibit
a basis of $\gl_{\BZ}''$ dual to our Chevalley basis ${\mathcal
C}'$, with respect to $\langle\,\ , \ \rangle$. Thus, we will show
that the pairing $\langle\,\ , \ \rangle^*_{\k}$ is non-degenerate.
Let $\mathfrak{t}'$ be the Cartan subalgebra of $\gl'$ spanned by
$\{ H_{\alpha}\,|\,\, \alpha \in \Pi \}$. Let $\{
H'_{\alpha}\,|\,\,\alpha \in \Pi \}$ be the dual basis of
$\mathfrak{t}'$ with respect to the restriction of $\langle\ , \
\rangle_\C$ to $\mathfrak{t}'$. (These may be thought of as the
fundamental weights of the dual root system $\Sigma^{\vee}$.) This
extends to a basis
\begin{equation*} {\mathcal C}\,=\, \big\{ H'_{\alpha}\,|\,\, \alpha \in \Pi
\big\}\,\, \textstyle{\bigsqcup}\,\, \big\{ X_\beta\,|\,\, \beta \in
\Sigma \mbox{ long} \big\}\,\, \textstyle{\bigsqcup}\,\,
\big\{(1/d)X_\beta\,|\,\, \beta \in \Sigma \mbox{ short} \big\}
\end{equation*}
of $\gl$ which is dual to our Chevalley basis ${\mathcal C}'$ with
respect to $\langle\,\ , \ \rangle_\C$. Specifically, the
corresponding pairing of basis elements is as follows:
\begin{eqnarray*}
H_{\alpha}   & \leftrightarrow & H'_{\alpha} \qquad\qquad\ \ \,\, \mbox{ if } \alpha\in\Pi,\\
X_{\beta} & \leftrightarrow &
X_{- \beta}\qquad\qquad\ \,\mbox{ if } \beta\in \Sigma \mbox{ is long}, \\
X_{\beta} & \leftrightarrow & (1/d)X_{ - \beta} \qquad \,\mbox{ if }
\beta\in\Sigma \mbox{ is short}.
\end{eqnarray*}
Moreover, it is easy to check that ${\mathcal C}$ is a $\BZ$-basis
of $\gl_{\BZ}''$, as required. Since the lattice $\gl_{\BZ}''$ is
admissible, we see that the bases ${\mathcal C}'\otimes 1$ of
$\gl=\gl'_\BZ\otimes_\BZ\k$ and ${\mathcal C}\otimes 1$ of
$\gl''_\BZ\otimes_\BZ\k$ are dual to each other with respect to
$\langle\ \,, \ \rangle^*_{\k}$. This shows that $\gl$ and
$\gl^*\cong\,\gl''_\BZ\otimes_\BZ\k$ are admissible $G$-modules
associated with different admissible lattices in $\gl'$.

Now suppose that $G$ is semisimple and simply connected. Then $G$ is
a direct product of simple, simply connected groups and the above
arguments carry over to $G$ in a straightforward fashion. In
particular, (\ref{form}) is still available for a suitable choice of
an admissible lattice $\gl''_\BZ\subset\gl'$ and
$\gl^*\cong\,\gl_\BZ''\otimes_\BZ\k$ as $G$-modules.

\begin{theorem} \label{g-strat}
Let $G$ be a connected reductive group over an algebraically closed
field $\k$ of characteristic $p\ge 0$ and let $\mathcal G$ be $\gl$
or $\gl^*$. If $\k$ is an algebraic closure of $\F_p$, assume
further that we have a Frobenius endomorphism $F\colon\,G\rightarrow
G$ corresponding to an $\F_q$-rational structure of $G$. Then
$\gP_1$--\,\,$\gP_5$ hold for $\mathcal G$ and the stabiliser $G_x$
of any element $x\in X^\wtri(\mathcal{G})$ is contained in the
parabolic subgroup $G_0^\wtri$ of $G$.
\end{theorem}
\begin{proof}
Let $U$ be an $F$-stable maximal connected unipotent subgroup of
$G$. It follows from the Hilbert--Mumford criterion (our
Theorem~\ref{HMum}) that $\N_\gl=\,(\Ad G)\cdot\mathfrak{u}$ where
$\mathfrak{u}=\Lie U$. Since $U\subset \mathcal{D}G$, we have that
$\N_\gl\subseteq \N_{\bar{\gl}}$ where
$\bar{\gl}=\Lie \mathcal{D}G$. As any $\xi\in\N_{\gl*}$ vanishes on
a Borel subalgebra of $\gl$, the restriction map
$\gl^*\rightarrow\,\bar{\gl}^*$, $\xi\mapsto \xi\vert_{\bar{\gl}},$
induces a $G$-equivariant injection
$\eta\,\colon\,\,\N_{\gl^*}\rightarrow\,\N_{\overline{\gl}^*}$. But
$\eta$ is, in fact, a bijection since every linear function on
$\mathfrak{u}$ can be extended to a nilpotent linear function on
$\gl$.

Let $\tilde{G}$ be a semisimple, simply connected group isogeneous
to $\mathcal{D}G$. Let $\iota\colon\,
\tilde{G}\rightarrow\,\mathcal{D}G$ be an isogeny and let
$\tilde{U}$ be the connected unipotent subgroup of $\tilde{G}$ with
$\iota(\tilde{U})=U$. Let $\tilde{\gl}=\Lie \tilde{G}$ and
$\tilde{\mathfrak{u}}=\Lie \tilde{U}$. Then ${\rm
d}_e\iota\colon\,\tilde{\gl}\rightarrow\,\bar{\gl}$ maps
$\tilde{\mathfrak{u}}$ isomorphically onto $\mathfrak{u}$ and
 induces a $\tilde{G}$-equivariant bijection
between $\N_{\tilde{\gl}}$ and $\N_{\bar{\gl}}=\bar{\gl}_{\rm nil}$. Let
$\tilde{T}$ be a maximal torus of $\tilde{G}$ normalising
$\tilde{\mathfrak{u}}$ and $T=\iota(\tilde{T})$, a maximal torus of
$G$ normalising $\mathfrak{u}$. We regard $\mathfrak{u}^*$ and
$\tilde{\mathfrak{u}}^*$ as subspaces of $\bar{\gl}^*$ and
$\tilde{\gl}^*$ respectively, by imposing that every
$\xi\in\mathfrak{u}^*$ vanishes on the $T$-invariant complement of
$\mathfrak{u}$ in $\gl$ and every
$\tilde{\xi}\in\tilde{\mathfrak{u}}^*$ vanishes on the
$\tilde{T}$-invariant complement of $\tilde{\mathfrak{u}}$ in
$\tilde{\gl}$. Then the linear map $({\rm d}_e\iota)^*\colon\,
\bar{\gl}^*\rightarrow\,\tilde{\gl}^*$ induced by ${\rm d}_e\iota$
restricts to a linear isomorphism between $\mathfrak{u}^*$ and
$\tilde{\mathfrak{u}}^*$. Since the map $({\rm d}_e\iota)^*$ is
$\tilde{G}$-equivariant, it induces a natural bijection between
$\N_{\tilde{\gl}^*}=\,(\Ad^*\tilde{G})\cdot\tilde{\mathfrak{u}}^*$
and $\N_{\gl^*}=\,(\Ad^* G)\cdot\mathfrak{u}^*$. It is clear from
our description of $F$ in Subsection~\ref{finitefield} that there is
a Frobenius endomorphism $\tilde{F}\colon\,\tilde{G}\rightarrow
\tilde{G}$ such that $\iota\circ \tilde{F}=F\vert_{\mathcal{D}G}$.
Furthermore, $\tilde{T}$ and $\tilde{U}$ can be chosen to be
$\tilde{F}$-stable.

The above discussion shows that in proving the theorem we may assume
that the group $G$ is semisimple and simply connected. Then both
$\gl$ and $\gl^*$ are admissible $G$-modules. More precisely,
$\gl=\,\gl_\BZ\otimes_\BZ\k$ and $\gl^*=\,\gl_\BZ''\otimes_\BZ\k$
for some admissible lattices $\gl_\BZ'$ and $\gl_\BZ''$ in $\gl'$.
Then Theorem~\ref{modulestrata} shows that the subsets
$H^{\btri}(\mathcal G)$ ($\btri \in D_G/G$) are the Hesselink strata
of $\N_{\mathcal G}$ and for each $\btri\in D_G/G$ the subsets
$X^\wtri(\mathcal G)$ with $\wtri\in\btri$ are the blades of
$\N_{\mathcal G}$ contained in $H^\btri(\mathcal G)$. In particular,
$\N_{\mathcal G}=\,\bigsqcup_{\wtri \in D_G}\,X^{\wtri}(\mathcal
G)$, showing that $\mathfrak{P}_3$ holds for $\mathcal G$. It
follows from \cite[Proposition~4.5]{Hess} that for every $\btri\in
D_G/G$ there is a surjective $G$-equivariant map
$H^\btri\twoheadrightarrow \, G/G_0^\wtri$ whose fibres are exactly
the blades $X^\wtri$ with $\wtri\in\btri$ (this map is not a
morphism, in general). So $\mathfrak{P}_1$ and $\mathfrak{P}_2$ hold
for $\mathcal G$ as well. In order to show that $\mathfrak{P}_4$
holds for $\mathcal G$ it suffices to establish that for every $x\in
X^\wtri(\mathcal G)$ the optimal parabolic subgroup $P(x)$ coincides
with $G_\wtri^0$. This is completely analogous to our arguments at
the end of the proof of Theorem \ref{main}. Of course it is much
easier since we may use Tsujii's result (Theorem \ref{Tsujii}) in
its original form, and there is no need for Section \ref{KirNess}.
The inclusion $G_x\subset G_\wtri^0$ follows from
Theorem~\ref{KRthm}(iv).

It remains to show that $\mathfrak{P}_5$ holds for $\mathcal G$, so
suppose from now on that $\k$ is an algebraic closure of $\F_p$ and
$F=F(\tau, l)$ where $q=p^l$; see Subsection~\ref{finitefield}. As
explained there, we have a natural $q$-linear action of $F$ on
$\gl^*$ compatible with the coadjoint action of $G$. We adopt the
notation introduced in the course of proving Theorem~\ref{NFq}. It
follows from Theorem~\ref{Kraft} that the set $\Lambda(\gl,
\tau)\,=\,\Lambda(\gl^*, \tau)$ consists of all pairs
$\big(\lambda'_{\scriptstyle{\btri}},k\big)$ such that
$\lambda'_\btri\in Y^+(T)$ is primitive, $k\in\{1,2\}$ and
$\frac{2}{k}\lambda'_\btri$ is adapted by a suitable nilpotent
element in the adjoint $G'$-orbit labelled by $\btri$. Then
(\ref{HFq}) yields
\begin{eqnarray*}
\varphi_\mathcal{G}^{\btri}(q)&:=&|H^{\btri}(\mathcal
G)^F|\,=\,f_{\tau,\lambda'_\btri}(q)\cdot q^{N(\lambda'_\btri,\,k)}
\Big(q^{n(\lambda'_\btri,k)}-|{\N_{\mathcal{G}(\lambda'_\btri,k)}}^F|\Big)\\
&=&f_{\tau,\lambda'_\btri}(q)\cdot q^{N(\lambda'_\btri,\,k)}
\Big(q^{n(\lambda'_\btri,k)}-n_{\mathcal{G}(\lambda'_\btri,k)}(q)\Big).
\end{eqnarray*}
If $\wtri\in\btri$ is such that $F(G_i^{\wtri}) = G_i^{\wtri}$ for
all $i \ge 0$, then the proof of Theorem~\ref{NFq} also yields that
$\tau^*(\lambda'_\btri)=\lambda'_\btri$ and
$$\psi_\mathcal{G}^{\wtri}(q):=\,|X^{\wtri}(\mathcal{G})^F|\,=\,q^{N(\lambda'_\btri,\,k)}
\Big(q^{n(\lambda'_\btri,k)}-n_{\mathcal{G}(\lambda'_\btri,k)}(q)\Big).$$
As the $L^\perp(\lambda'_\btri)$-modules $\gl(\lambda'_\btri,k)$ and
$\gl^*(\lambda'_\btri,k)$ come from different admissible lattices of
the $(\mathfrak{L}^\perp(\lambda'_\btri))(\C)$-module
$\gl'(\lambda'_\btri,k)$, applying Theorem~\ref{NFq} shows that
$\psi_\gl^\wtri(q)\,=\,\psi_{\gl^*}^\wtri(q)$ are polynomials in $q$
with integer coefficients independent of $p$. This, in turn, implies
that so are $\varphi_\gl^\btri(q)\,=\,\varphi_{\gl^*}^\btri(q)$,
completing the proof.
\end{proof}

\begin{corollary}\label{last}
Let $G$ be a connected reductive group defined over an algebraic
closure of $\F_p$ and assume that we have a Frobenius endomorphism
$F\colon\,G\rightarrow G$ corresponding to an $\F_q$-rational
structure on $G$. Then $\gP_5$ holds for $G$.
\end{corollary}
\begin{proof}
Let $\wtri\in D_G$ be such $F(G_i^\wtri)=G_i^\wtri$ for all $i\ge 0$
and let $\btri$ be the orbit of $\wtri$ in $D_G/G$. Then $gG_0^\wtri
g^{-1}\,=\,P(\lambda_\btri')$  and $gG_i^\wtri
g^{-1}=\,U_i(\lambda_\btri')$ for some $g\in G$, where $i\ge 1$. If
$s$ is the order of $\tau^*$, then there exists $r\in\BN$ with
$r\equiv 1\ ({\rm mod}\ s)$ such that ${X^\wtri(G)}^{F^r}\ne
\emptyset$. Then ${H^\btri(G)}^{F^r}\ne \emptyset$ and the argument
used in the proof of Theorem~\ref{NFq} shows that
$\tau^*(\lambda_\btri')={\tau^*}^r(\lambda_\btri^*)=\lambda_\btri'$.
Since $\tau^*(\lambda_\btri')={\tau^*}^r(\lambda_\btri^*)$ by our
choice of $r$, we see that $P(\lambda_\btri')$ is $F$-stable. Hence
$g^FG_0^\wtri(g^F)^{-1}=\,gG_0g^{-1}$ forcing $g^{-1}g^F\in
N_G(G_2^\wtri)=\,G_0^\wtri$. As $G_0^\wtri$ is connected and
$F$-stable, the Lang--Steinberg theorem shows that
$g^{-1}g^F=x^{-1}x^F$ for some $x\in G_0^\wtri$; see
\cite[Theorem 3.10]{DM}. Replacing $g$ by $gx^{-1}$ we thus may assume that
$g\in G^F$. In conjunction with Theorems~\ref{KNuni} and
\ref{uni-pieces} this shows that
\begin{equation}\label{A}\big|X^\wtri(G)^F\big|=\,
\big|\pi^{-1}\big({V_2(\lambda_\btri')_{ss}}^F\big)\big|
\end{equation}
where $V_2(\lambda_\btri')_{ss}$ stands for the set of all
$L^\perp(\lambda_\btri')$-semistable vectors of the
$L(\lambda_\btri')$-module
$V_2(\lambda_\btri')=\,U_2(\lambda_\btri')/U_3(\lambda_\btri')$ and
$\pi\colon\,U_2(\lambda_\btri')^F\rightarrow\,V_2(\lambda_\btri')^F$
is the map induced by the canonical homomorphism
$U_2(\lambda_\btri')\twoheadrightarrow\,V_2(\lambda_\btri')$. Now
the argument used in the proof of Theorem~\ref{NFq} yields
\begin{equation}\label{B}\big|H^\btri(G)^F\big|\,=\,
\big|G^F/P(\lambda_\btri')^F)\big|\cdot
\big|\pi^{-1}\big({V_2(\lambda_\btri')_{ss}}^F\big)\big|.\end{equation}
In view of Remark~\ref{overZ} we have that
\begin{equation}\label{C}\big|{V_2(\lambda_\btri')_{ss}}^F\big|=\,\big|{\gl(\lambda_\btri',\,2)_{ss}}^F\big|.
\end{equation}
Since the group $U_3(\lambda_\btri')$ is connected and $F$-stable,
the Lang--Steinberg theorem shows that for every
$v\in{V_2(\lambda_\btri')_{ss}}^F$ there is an element $\tilde{v}\in
{V_2(\lambda_\btri')_{ss}}^F$ such that $\pi(\tilde{v})=v$. From
this it is immediate that
\begin{equation}\label{D}
\pi^{-1}(v)=\,\tilde{v}\cdot
U_3(\lambda_\btri')^F\qquad\quad\big(\forall\,v\in{V_2(\lambda_\btri')_{ss}}^F\big).
\end{equation}
Combining (\ref{A}), (\ref{C}) and (\ref{D}) we obtain that
\begin{equation}\label{E}
\big|X^\wtri(G)^F\big|=\,\big|\pi^{-1}\big({V_2(\lambda_\btri')_{ss}}^F\big)\big|
=\,\big|{\gl(\lambda_\btri',\,2)_{ss}}^F\big|\cdot\big|U_3(\lambda_\btri')^F\big|.
\end{equation}
As we know by Remark~\ref{overZ}, for each $i\ge 3$ the connected
abelian group
$V_i(\lambda_\btri')=\,U_i(\lambda_\btri')/U_{i+1}(\lambda_\btri')$
is a vector space over $\k$  isomorphic to $\gl(\lambda_\btri,i)$.
Since $\tau^*\lambda_\btri'=\lambda_\btri'$, it is equipped with a
$q$-linear action of $F$. Therefore
\begin{equation}\label{F}|V_i(\lambda_\btri')^F|=\,q^{\dim \gl(\lambda_\btri',\,i)},\quad\ i\ge 3;\end{equation}
see \cite[Corollary~3.5]{DM}, for example. Since every group
$U_i(\lambda_\btri')$ with $i\ge 3$ is connected and $F$-stable, the
Lang--Steinberg theorem yields that for every $u\in
V_i(\lambda_\btri')^F$ there exists $\tilde{u}\in
U_i(\lambda_\btri')^F$ whose image in $V_i(\lambda_\btri')^F$ equals
$u$. This, in turn, implies that every quotient
$V_i(\lambda_\btri')^F$ with $i\ge 3$ has a section in
$U_i(\lambda_\btri')^F$; we call it
$\widetilde{V}_i(\lambda_\btri')$. Then
\begin{equation}\label{G}
\big|U_3(\lambda_\btri')^F\big|=\,\prod_{i\ge
3}\,\big|\widetilde{V}_i(\lambda_\btri')^F\big|.
\end{equation}
Together (\ref{E}), (\ref{F}) and (\ref{G}) show that
$$
\big|X^\wtri(G)^F\big|=\,\big|\pi^{-1}\big({V_2(\lambda_\btri')_{ss}}^F\big)\big|=\,\big(q^{\dim
\gl(\lambda_\btri',\,2)}-\big|{\N_{\gl(\lambda_\btri',\,2)}}^F\big|\big)\cdot
q^{\dim\gl(\lambda_{\btri}',\,\ge 3)}.$$ As a result,
$|X^\wtri(G)^F|=\,|X^\wtri(\gl)^F|=\,\psi_\gl^\wtri(q)$ for every
$\wtri$ as above. Now (\ref{B}) yields
$|H^\btri(G)^F|=\,|H^\btri(\gl)^F|=\,\varphi_\gl^\btri(q)$. In view
of Theorem~\ref{g-strat} this implies that $\mathfrak{P}_5$ holds
for $G$.
\end{proof}
\begin{rmk}\label{L-X}
1. In the appendix to \cite{Lus7} and more recently in \cite{Lus1},
Lusztig and Xue proposed for $G$ classical a definition of nilpotent
pieces which avoids the partial ordering of nilpotent orbits. Given
$\wtri\in D_G$ choose $g\in G$ as in Subsection~\ref{morepieces} and
define $\gl_2^{\wtri !}$ to be the set of all $x=\sum_{i\ge 2}x_i\in
\gl_2^{\wtri}$ with $x_i\in\gl(i, g\cdot\omega)$ and
$C_G(x_2)\subset G_2^\wtri$. Similarly, let $(\gl^*)_2^{\wtri !}$ be
the set of all $\xi=\sum_{i\ge 2}\xi_i\in (\gl^*)_2^{\wtri}$ with
$\xi_i\in\gl^*(i, g\cdot\omega)$ such that the stabiliser of $\xi_2$
in $G$ is contained in $G_0^\wtri$. According to the definition of
Lusztig and Xue, the nilpotent pieces of $\gl$ and $\gl^*$ are
$$\big\{\gl_2^{\wtri !}\,|\,\, \wtri\in D_G\big\}\, \,\mbox { and }\, \,
\big\{(\Ad G)\cdot\gl_2^{\wtri !}\,|\,\, \btri\in D_G/G \big\}$$ and
$$\big\{(\gl^*)_2^{\wtri !}\,|\,\, \wtri\in D_G\big\}\, \,\mbox{ and
}\, \, \big\{(\Ad^* G)\cdot(\gl^*)_2^{\wtri !}\,|\,\, \btri\in D_G/G
\big\},$$ respectively, where $\wtri$ is implicitly taken to be a
representative of $\btri$ in each case. Lusztig and Xue proved that
for $G$ classical these subsets stratify $\N_\gl$ and $\N_{\gl^*}$.
On the other hand, Theorem~\ref{g-strat} implies that
$X^\wtri(\gl)\subseteq \gl_2^{\wtri !}$ and $X^\wtri(\gl^*)\subseteq
(\gl^*_2)^{\wtri !}$ for every $\wtri\in D_G$.  But equality must
hold in each case because the blades, too, stratify the nullcones.
This shows that for $G$ classical both definitions lead to the same
stratifications of $\N_\gl$ and $\N_{\gl^*}$.

\smallskip

\noindent 2. The proof of Corollary~\ref{last} shows that for any
$p>0$ there exists a bijection between ${G_{\rm uni}}^F$ and
${\gl_{\rm nil}}^F$ which maps every non-empty subset $X^\wtri(G)^F$
onto $X^\wtri(\gl)^F$ and every non-empty subset $H^\btri(G)^F$ onto
$H^\btri(\gl)^F$.

\smallskip

\noindent 3. It follows from \cite[Proposition~6(2)]{Sesh} that for
every $\btri\in D_G/G$ there is a homogeneous regular function
$f_\btri\in \Z[\gl_\BZ'(\lambda_\btri',2)]$ invariant under the
natural action of the group scheme
$\mathfrak{L}^\perp(\lambda_\btri')$ and such that for any
algebraically closed field $\k$ the variety
$\N_{\gl(\lambda_\btri',\,2)}$ coincides with the zero locus of the
image of $f_\btri$ in
$\k[\gl(\lambda_\btri',2)]=\,\Z[\gl_\BZ'(\lambda_\btri',2)]\otimes_\BZ\k$;
see \cite[\S 2.4]{Pre} for a related discussion.
\end{rmk}

\end{document}